\documentclass{article}
\usepackage[T1]{fontenc} 
\usepackage[latin1]{inputenc}
\usepackage{amsmath,amsfonts,amssymb,amsthm,psfrag, color,wrapfig,float,dsfont,stmaryrd,hyperref}
\usepackage{authblk}
\usepackage{xcolor}

\usepackage[a4paper,centering]{geometry}

\usepackage[french,english]{babel}
\title{Reducing exit-times of diffusions with repulsive interactions}
\author[1]{Paul-Eric Chaudru de Raynal}
\author[2]{Manh Hong Duong}
\author[3]{Pierre Monmarch\'e}
\author[4]{Milica Toma\v sevi\'c}
\author[5]{Julian Tugaut}
\affil[1]{Nantes Universit\'e, CNRS, Laboratoire de Math\'ematiques Jean Leray, LMJL, F-44000 Nantes, France}
\affil[2]{School of Mathematics, University of Birmingham, B15 2TT Birmingham, UK}
\affil[3]{LJLL and LCT, Sorbonne Universit\'e, 4 place Jussieu, 75005 Paris, France}
\affil[4]{CNRS, CMAP, Ecole Polytechnique, I.P. Paris, Route de Saclay 91128 Palaiseau, France. }
\affil[5]{Universit\'e Jean Monnet, Institut Camille Jordan, 23, rue du docteur Paul Michelon,
CS 82301,
42023 Saint-\'Etienne Cedex 2,
France.}

\begin{document}

\newcommand{\bRb}{\mathbb{R}}
\newcommand{\bCb}{\mathbb{C}}
\newcommand{\bEb}{\mathbb{E}}
\newcommand{\bKb}{\mathbb{K}}
\newcommand{\bQb}{\mathbb{Q}}
\newcommand{\bFb}{\mathbb{F}}
\newcommand{\bGb}{\mathbb{G}}
\newcommand{\bNb}{\mathbb{N}}
\newcommand{\bZb}{\mathbb{Z}}

\newcommand{\deriv}{\stackrel{\mbox{\bf\Large{}$\cdot$\normalsize{}}}}
\newcommand{\dederiv}{\stackrel{\mbox{\bf\Large{}$\cdot\cdot$\normalsize{}}}}

\theoremstyle{plain} \newtheorem{thm}{Theorem}[section]
\theoremstyle{plain} \newtheorem{prop}[thm]{Proposition}
\theoremstyle{plain} \newtheorem{props}[thm]{Properties}
\theoremstyle{plain} \newtheorem{ex}[thm]{Example}
\theoremstyle{plain} \newtheorem{contrex}[thm]{Coounterexample}
\theoremstyle{plain} \newtheorem{cor}[thm]{Corollary}
\theoremstyle{plain} \newtheorem{hyp}[thm]{Hypothesis}
\theoremstyle{plain} \newtheorem{assu}[thm]{Assumption}
\theoremstyle{plain} \newtheorem{hyps}[thm]{Hypotheses}
\theoremstyle{plain} \newtheorem{lem}[thm]{Lemma}
\theoremstyle{plain} \newtheorem{rem}[thm]{Remark}
\theoremstyle{plain} \newtheorem{nota}[thm]{Notation}
\theoremstyle{plain} \newtheorem{defn}[thm]{Definition}

\newcommand{\cRc}{\mathcal{R}}
\newcommand{\cCc}{\mathcal{C}}
\newcommand{\cEc}{\mathcal{E}}
\newcommand{\cKc}{\mathcal{K}}
\newcommand{\cQc}{\mathcal{Q}}
\newcommand{\cFc}{\mathcal{F}}
\newcommand{\cGc}{\mathcal{G}}
\newcommand{\cNc}{\mathcal{N}}
\newcommand{\cZc}{\mathcal{Z}}
\newcommand{\cOc}{\mathcal{O}}
\newcommand{\cSc}{\mathcal{S}}
\newcommand{\cAc}{\mathcal{A}}
\newcommand{\cBc}{\mathcal{B}}
\newcommand{\cIc}{\mathcal{I}}
\newcommand{\cDc}{\mathcal{D}}
\newcommand{\cLc}{\mathcal{L}}
\newcommand{\cHc}{\mathcal{H}}

\newcommand{\Ima}{{\rm Im}}
\newcommand{\Rea}{{\rm Re}}
\newcommand{\gaga}{\left|\left|}
\newcommand{\drdr}{\right|\right|}
\newcommand{\lra}{\left\langle}
\newcommand{\rra}{\right\rangle}

\newcommand{\bal}{\begin{align}}
\newcommand{\eal}{\end{align}}
\newcommand{\beq}{\begin{equation}}
\newcommand{\eeq}{\end{equation}}

\newcommand{\ba}{\begin{align*}}
\newcommand{\be}{\begin{equation*}}
\newcommand{\ee}{\end{equation*}}

\newcommand{\EE}{\mathbb{E}}
\newcommand{\PP}{\mathbb{P}}

\newcommand{\sepa}{\left|\right.}

\newcommand{\crg}{[\![}
\newcommand{\crd}{]\!]}

\newcommand{\sgn}{{\rm Sign}}
\newcommand{\vari}{{\rm Var}}
\newcommand{\cov}{{\rm Cov}}
\newcommand{\poin}[1]{\dot{#1}}
\newcommand{\norm}[1]{\Vert #1\Vert}
\def\div{\mathop{\mathrm{div}}\nolimits}

\newcommand{\p}{\partial}
\newcommand{\x}{\partial_x}
\newcommand{\y}{\partial_y}
\newcommand{\z}{\partial_z}
\newcommand{\po}{\left(}
\newcommand{\pf}{\right)}
\newcommand{\co}{\left[}
\newcommand{\cf}{\right]}
\newcommand{\cco}{\llbracket}
\newcommand{\ccf}{\rrbracket}
\newcommand{\R}{\mathbb R}
\newcommand{\N}{\mathbb N}
\newcommand{\T}{\mathbb T}
\newcommand{\na}{\nabla}
\newcommand{\dd}{\mathrm{d}}
\newcommand{\1}{\mathbbm{1}}
\newcommand{\rmq}[1]{{\color{blue}(#1)}}

\makeatletter
\newcommand{\customlabel}[2]{%
   \protected@write \@auxout {}{\string\newlabel {#1}{{#2}{\thepage}{#2}{#1}{}}}%
   \hypertarget{#1}{#2\hspace{-0.14cm}}
}
\makeatother

\newcommand{\cset}{\mathbb{C}}
\newcommand{\rset}{\mathbb{R}}
\newcommand{\qset}{\mathbb{Q}}
\newcommand{\zset}{\mathbb{Z}}
\newcommand{\nset}{\mathbb{N}}
\newcommand{\nl}{\nolimits}
\newcommand{\ep}{\varepsilon}
\newcommand{\ind}{\mathbf{1}}
\newcommand{\fl}{\longrightarrow}
\newcommand{\e}{\mathbb{E}}
\newcommand{\E}{\mathbb{E}}
\newcommand{\gX}{\mathbf{X}}
\renewcommand{\d}{\,d}
\newcommand{\pp}{\mathcal P_2}
\newcommand{\lp}{\mathrm{L}}
\newcommand{\m}{\mathcal} 
\newcommand{\mL}{\mathcal{A}} 
\newcommand{\mA}{\mathcal{L}} 
\newcommand{\bx}{{\bf x}}
\newcommand{\by}{{\bf y}}
\newcommand{\ys}[1][2]{\mathcal{S}^{#1}}
\newcommand{\zs}[1][2]{\mathrm{M}^{#1}}
\newcommand{\ld}[1][\mu]{\text{D}_{#1}}

\newcommand \A[1]{{\bf (#1)}}

\maketitle

\begin{abstract}
 In this work we prove a Kramers' type law for the low-temperature behavior of the exit-times  from a metastable state for a class of self-interacting  nonlinear diffusion processes. Contrary to previous works, the interaction is not assumed to be convex, which means that this result covers cases where the exit-time for the interacting process is smaller than the exit-time for the associated non-interacting process. The technique of the proof is based on the fact that, under an appropriate contraction condition, the interacting process is conveniently coupled with a non-interacting (linear) Markov process where the interacting law is replaced by a constant Dirac mass at the fixed point of the deterministic zero-temperature process. 
\end{abstract}
\medskip

 \par\medskip

 \par\medskip

\section{Introduction and main result}
\subsection{Motivations}
\label{sec:motivations}
In Monte Carlo Markov chain (MCMC) methods, high-dimensional expectations with respect to a given target probability distribution $\pi$, say on $\R^d$ with a Lebesgue density proportional to $\exp(-V)$ for some potential $V$, are estimated thanks to an ergodic law of large numbers:
\begin{equation}\label{eq:LGNergodic}
    \frac1t\int_0^t \varphi(X_s) \dd s \ \underset{t\rightarrow +\infty}\longrightarrow \ \int_{\R^d} \varphi(x) \pi(\dd x)\,,
\end{equation}
where $(X_t)_{t\geqslant0}$ is a Markov process designed to be ergodic with respect to $\pi$, and $\varphi$ is some observable of interest. A classical example is the overdamped Langevin diffusion
\begin{equation*}
\dd X_t \ = \ -\na V(X_t) \dd t + \sqrt 2 \dd B_t\,,
\end{equation*}
where $(B_t)_{t\geqslant 0}$ is a standard Brownian motion, see e.g. \cite{ActaNumericaLS}. The estimation given by  the convergence \eqref{eq:LGNergodic} is reasonable if $t$ is sufficiently large so that the process has visited during $[0,t]$ a representative sample (with respect to $\pi$) of the state space. In particular, if $\pi$ is multimodal, i.e. if $V$ has several local minima, then $t$ should be at least large enough so that some transitions have occurred between the basins of attraction of the main minima (where \emph{main} means: in term of contribution to the expectation). In the theoretical studies of this question, the difficulty of the problem is measured by the addition of a temperature parameter $\varepsilon>0$, the target density being proportional to $\exp(-V/\varepsilon)$, so that the multimodality worsens as $\varepsilon$ vanishes (i.e. at low temperature). The corresponding overdamped Langevin process is then
\begin{equation}
\label{eq:overdampedLangevin2}
\dd X_t \ = \ -\na V(X_t) \dd t + \sqrt {2\varepsilon} \dd B_t\,.
\end{equation}
As $\varepsilon$ vanishes, this is essentially a gradient descent, which means it converges quickly to a local minimum, and then has to wait a large deviation of the small Brownian noise to escape to another basin. As a consequence, transitions take a time that is exponentially large with $\varepsilon^{-1}$, which makes the convergence \eqref{eq:LGNergodic} very slow. This is a so-called metastable behaviour \cite{LelievreMetastable}. What is true for the overdamped Langevin diffusion is in fact generic for all classical samplers: indeed, in practice, we only have a local knowledge of the energy landscape, which prevents the design of Markov processes that perform large jumps. Standard samplers have thus continuous trajectories, or perform small jumps (the exploration is local). This implies that, in practice, during a transition from one basin to another, the time spent in the low-probability  area in between is lower bounded uniformly with $\varepsilon$. But then the ergodic property \eqref{eq:LGNergodic} implies that the ratio of the times spent in two areas tends to the ratio of the probabilities of those. As a consequence, we get that, roughly speaking, at low temperature,  for a Markov process ergodic with respect to $\pi$ and performing a local exploration, the time between transitions is in some sense necessarily of order at least $\exp( c/\varepsilon)$ where $c$ is the difference of energy between the local minima and the lowest saddle point in between.

Besides, this issue of metastability in the exploration of a high-dimensional non-convex energy landscape also arises for optimization problems. Indeed, consider the process
\[\dd X_t \ = \ -\na V(X_t) \dd t + \sqrt {2\varepsilon_t} \dd B_t\,,\]
where $t\mapsto \varepsilon_t$ is a decreasing vanishing map called the cooling schedule. This is a simple theoretic simulated annealing process. It is well-known that, provided $\varepsilon_t$ vanishes sufficiently slowly with $t$, then the process converges in probability to the global minima of $V$ \cite{HolleyKusuoakaStroock}. The picture is very similar to the question of the convergence of an ergodic mean at constant (small) temperature: indeed, the process needs sufficient time to cross energy barriers and find global minima. In both cases (sampling and optimization), the real problem is exploration (discovering the main basins of attraction).

As we saw, at low temperature, exits from local minima are rare events. A general method to tackle rare events issues is importance sampling: the target $\pi$ is replaced by a biased target $\tilde \pi \propto \exp (-(V-A)/\varepsilon)$ for some biasing potential $A$, in such a way the biased process
\[\dd X_t \ = \ -\na (V-A)(X_t) \dd t + \sqrt {2\varepsilon} \dd B_t\,,\]
is less metastable than the initial process, so that the convergence
\[\frac{\int_0^t \varphi(X_s)e^{-A(X_s)/\varepsilon}\dd s }{\int_0^t e^{-A(X_s)/\varepsilon}\dd s } = \frac{t^{-1}\int_0^t \varphi(X_s)e^{-A(X_s)/\varepsilon}\dd s }{t^{-1}\int_0^t e^{-A(X_s)/\varepsilon}\dd s } \underset{t\rightarrow+\infty}\longrightarrow \frac{\int_{\R^d} \varphi(x) e^{-A(x)/\varepsilon}\tilde \pi(x)}{\int_{\R^d}  e^{-A(x)/\varepsilon}\tilde \pi(x)} = \int_{\R^d} \varphi(x)\pi(\dd x)\]
is faster than \eqref{eq:LGNergodic} ($\tilde \pi$ should still reasonably close to $\pi$, otherwise the weights $e^{-A}$ in the estimator induce a large variance). Designing a good biasing potential $A$ is a difficult question. Starting from a given local minimum $a$ of $V$, a simple idea would be to take $A(x) = \alpha|x-a|^2$ for some $\alpha>0$, which would tend to repel the particle away from $a$. But this requires the knowledge of $a$; and then when the process has moved to another basin, another local minimum would have to be considered.

For this reason, in fact, as far as metastability is concerned, many strategies are based on \emph{adaptive} biasing potentials, i.e. $\nabla A$ is not fixed a priori but evolves in time, depending of some current knowledge of the energy landscape. The basic idea is the following: if the process has already spent a long time in some area, then we should add a repulsion force from this area to accelerate the escape. In many cases (see e.g. \cite{LelievreFreeEnergy,Fortetal,LelievreABF,BBM2020,Metadyn2} and references within) these adaptive biasing forces can be viewed as interactions with some probability law, i.e. the process is 
\begin{equation}\label{eq:non-linear}
    \dd X_t \ = \ -\nabla (V-A_{\mu_t})(X_t) \dd t + \sqrt {2\varepsilon} \dd B_t\,,
\end{equation}
where $\mu_t$ is a probability measure: typically, the occupation measure of the past trajectory $(X_s)_{s\in[0,t]}$ as in \cite{BBM2020}, or the empirical measure of a system of $N$ particles. If particles are repelled one from the other, the system will cover a larger area, enhancing the exploration. As $N$ goes to infinity, according to the propagation of chaos phenomenon, the particles become independent and the empirical measure of the system converges to the law of one of the particle, which leads to non-linear processes  $X$ that solves \eqref{eq:non-linear} with $\mu_t$ the law of $X_t$.

This leads  to the question addressed in the present work, that is the question of reducing, through self-interaction, exit-times from the vicinity of local minima of $V$ at low temperature.

\subsection{Framework and strategy}

Let $M \in \R^{d \otimes d}$, $a:\R^d \to \R^d$ and $\sigma \ge 0$. As discussed before, our aim consists in decreasing the exit-time $\mathcal T_\sigma(\mathcal{D}):=\inf\{t\geqslant 0\ : \  Y_t^\sigma \notin\mathcal D\}$ from a  domain $\mathcal D$ of the diffusion process that solves
\begin{equation}\label{linear_convex}
\dd Y_t^\sigma \ = \ a(Y_t^\sigma) \dd t  + \sigma M\dd B_t\,.
\end{equation}
For instance, the overdamped Langevin process \eqref{eq:overdampedLangevin2} corresponds to $a = -\na U$ and $M=\sqrt 2 I$. Besides, the general form \eqref{linear_convex} covers various other cases of interest, like the kinetic Langevin process or coloured noise processes, as  detailed in Section~\ref{subsec:ex-linear-drift} below.

It is  known that under suitable conditions, such an exit-time satisfies a so-called Kramers' type law: there exists $L>0$ (given as the solution of a variational problem) such that,  
for all $\delta>0$,
\begin{equation}
\label{kramer_linear}
\lim_{\sigma\to0}\PP\left\{\exp\left(\frac{2}{\sigma^2}(L-\delta)\right)\leq\mathcal T_\sigma(\mathcal{D})\leq\exp\left(\frac{2}{\sigma^2}(L+\delta)\right)\right\}=1\,.
\end{equation}

We are interested in modifications of \eqref{linear_convex} consisting in processes whose evolution at time $t$ depends on a given (deterministic) probability law $\mu_t^\sigma$. Denoting by $m_t^\sigma$ the law of the process at time $t$, the most classical case would be $\mu_t^\sigma = m_t^\sigma$, which gives a classical non-linear McKean-Vlasov diffusion. We have also in mind the case of so-called memorial McKean-Vlasov processes, where $\mu_t^\sigma = t^{-1}\int_0^t m_s^\sigma \dd s$. A convenient way to gather these two examples, and more generally memorial processes with a non-uniform memory kernel which can be motivated by applications in stochastic algorithms, is to consider $\mu_t^\sigma = \int_0^t m_s^\sigma   R(t,\dd s)$ where, for all $t\geqslant 0$, $R(t,\cdot)$ is a probability measure on $[0,t]$. The classical non-linear case then corresponds to $R(t,\cdot) = \delta_t$ for all $t\geqslant 0$, and the memorial process to the case where $R(t,\cdot)$ is the uniform law on $[0,t]$ for all $t\geqslant 0$.

Let $\mathcal P_2(\R^d)$ be the set of probability measures on $\R^d$ with a finite second moment and $\mathcal P(\R)$ the set of probability measures on $\R$. Given $b:\R^d \times \mathcal P_2(\R^d) \rightarrow \R^d$, an initial condition $m_0^\sigma \in \mathcal P_2(\R^d)$ and   $R:\R \rightarrow \mathcal P(\R)$, we  are thus interested in a process $(X_t^{\sigma})_{t\geqslant 0}$  with $X_0^{\sigma} \sim m_0^\sigma$ and such that, for all $t\geqslant 0$,
\begin{equation}\label{eq:init}
\left\{\begin{array}{rcl}
\dd X_t^{\sigma} &=& a(X_t^{\sigma}) \dd t +  b(X_t^{\sigma},\mu_t^\sigma) \dd t + \sigma M\dd B_t,\\
\mu^\sigma_t &:= & \int_0^t m_{s}^\sigma  R(t,\dd s)\qquad  m_t^\sigma : = \mathcal L(X_t^{\sigma}).
\end{array}\right.
\end{equation}
Conditions on $a$, $b$ and $R$ that ensures the existence and uniqueness in distribution of such a process will be discussed below. 
The main subject of this work is the exit-time 
\begin{equation}\label{sandra}
\tau_\sigma(\mathcal{D}):=\inf\left\{t\geq0\,\,:\,\,X_t^{\sigma}\notin\mathcal{D}\right\}
\end{equation}
of the above process from a given domain $\mathcal D$ on $\R^d$. We aim at proving a Kramers' law
\begin{equation}
\label{eq:KramerNonLin}
\forall \delta>0\,,\qquad \lim_{\sigma\to0}\PP\left\{\exp\left(\frac{2}{\sigma^2}(H-\delta)\right)\leq\tau_\sigma(\mathcal{D})\leq\exp\left(\frac{2}{\sigma^2}(H+\delta)\right)\right\}=1\,,
\end{equation}
for some $H>0$.

More precisely, our main contribution is Theorem \ref{th:mr} where we establish \eqref{eq:KramerNonLin} under conditions that cover cases where $H<L$ (with $L$ given in \eqref{kramer_linear}), which means that the exit-time is shorter for the interacting process than for the non perturbed diffusion \eqref{linear_convex}, contrary to the similar works \cite{EJP,ECP} which are restricted to convex interactions.

Our general strategy to establish \eqref{eq:KramerNonLin} is to prove, under some conditions   (see next section), the following. First, at a fixed $\sigma$, $\mu_t^\sigma$ converges in large time to  an equilibrium $\mu_\infty^\sigma$, at a speed \emph{that is uniform in} $\sigma$ (for the Wasserstein distance $\mathbb W_2$, see below). Second, as $\sigma$ vanishes, $\mu_\infty^\sigma$ converges to  $\delta_{\lambda}$ for some $\lambda \in \R^d$. Hence, the interacting process \eqref{eq:init} is expected to behave similarly to the  \emph{linear} (in the McKean-Vlasov sense) diffusion that solves
\begin{equation}
\label{roserouge}
\dd \tilde{X}_t^{\sigma} \ = \ a(\tilde{X}_t^{\sigma} ) \dd t +  b(\tilde{X}_t^{\sigma} ,\delta_{\lambda}) \dd t + \sigma M\dd B_t\,,
\end{equation}
for which the Kramers' law follows from classical results. In fact, more precisely, we can consider the two equations \eqref{eq:init} and \eqref{roserouge} simultaneously, driven by the same Brownian motion $(B_t)_{t\geqslant 0}$. A crucial point is then to prove that, at low temperature, the two processes will \emph{deterministically} stay  close one to the other, so that the exit of one of the process from an enlargement of $\mathcal D$ necessarily implies the exit of the other from $\mathcal D$.

\subsection{Assumptions and main results}\label{sec:assum_mainresult}

We divide the assumptions in three groups: the first group concerns the coefficients of the dynamics~\eqref{eq:init}, ensuring in particular the well-posedness of the process. In the second group we gather basic conditions on the domain and the initial condition. The third one states a Kramers' law for a linear process of the form~\eqref{roserouge}. 

We start by introducing the conditions \customlabel{A}{\A{A}}   that concern the drift functions $a$ and $b$ and the memory kernel $R$.

\begin{itemize}
\item[\customlabel{A8}{\A{A1}}] The function $a:\R^d\rightarrow \R^d$ (resp. $b:\R^d\times \mathcal P_2(\R^d) \rightarrow \R^d$) is locally (resp. globally) Lipschitz continuous.
\item[\customlabel{A9}{\A{A2}}] There exist $\rho>\kappa\geqslant 0$ such that for all $z,y\in\R^d$ and all $\mu,\nu\in\mathcal P_2(\R^d)$,
\begin{eqnarray}
\label{eq:a_contract}
(z-y) \cdot   \po a(z)+b(z,\mu)-a(y)-b(y,\nu) \pf & \leqslant & -\rho |z-y|^2 + \kappa \mathbb W_2^2(\nu,\mu)\,.
\end{eqnarray}
\item[\customlabel{A4}{\A{A3}}] For all $s\geqslant 0$, $R(t,[0,s]) \rightarrow 0$ as $t\rightarrow +\infty$ and for all continuous $f:\R\rightarrow\R$, $t\mapsto \int_0^t f(s)R(t,ds)$ is measurable.
\end{itemize}

Let us point out that $R(t,ds)=\delta_t(ds)$ or $R(t,ds)=\frac{1}{t}\mathds{1}_{[0;t]}(s)ds$ satisfy Assumption \ref{A4}. Moreover, \ref{A9} is true if $a$ satisfies \eqref{eq:a_contract} (applied with $b=0$, with $\kappa=0$) and  the Lipschitz constants of $b$, namely $\kappa_1,\kappa_2\geqslant 0$ such that
\begin{equation}
\label{pairet}
|b(z,\mu) - b(y,\nu)| \leqslant \kappa_1 |z-y| + \kappa_2 \mathbb W_2(\mu,\nu)
\end{equation}
for all $z,y\in\R^d$ and all $\mu,\nu\in\mathcal P_2(\R^d)$, satisfy $\kappa_1+\kappa_2<\rho$. However, \ref{A9} is weaker when the interaction is attracting, for instance if $b(x,\nu)=\alpha(\int_{\R^d} y \nu(dy) - x)$ with $\alpha>0$.

Let us deduce a few preliminary results from these first assumptions. First, we check that the interacting process is indeed well defined under \ref{A}.

\begin{prop}\label{prop:gen:wp} Under Assumption \ref{A}, the system \eqref{eq:init} admits a unique weak solution. 
\end{prop}

The proof is postponed to Section~\ref{sec:existenceprocess}.

Next, in order to state the conditions on the domain $\mathcal D$, we need the following lemma. 

\begin{lem}
\label{lem:exist_lambda}
Under \ref{A}, there exists a unique $\lambda \in \R^d$ such that 
\begin{equation}\label{eq:lambda}
a(\lambda)+b(\lambda,\delta_{\lambda})\ =\ 0\,.
\end{equation}
\end{lem}
The proof is postponed (see Lemma~\ref{cor:deterministic_flow}). Notice that \eqref{eq:lambda} is equivalent to saying that  the process given by $X_t=\lambda$ for all $t\geqslant 0$  is a constant solution of \eqref{eq:init} in the zero noise case $\sigma=0$. In all the rest of this section, we fix $\lambda$ as given by Lemma~\ref{lem:exist_lambda}.

The basic conditions \customlabel{D}{\A{D}} on the domain $\mathcal D$ are the following.

\begin{trivlist}
\item[\customlabel{D1}{\A{D1}}] The domain $\mathcal{D}\subsetneq\R^d$ is open.
\item[\customlabel{D3}{\A{D2}}] Moreover $a$ is Lipschitz continuous on $\mathcal D_{+1}=\{z\in\R^d,\ d(z,\mathcal D)\leqslant 1\}$.
\item[\customlabel{D4}{\A{D3}}] The initial distribution $m_0^\sigma=m_0$ is independent from $\sigma$ and has a compact support included in a ball centered at $\lambda$ (given in Lemma~\ref{lem:exist_lambda}) included in $\mathcal D$. \end{trivlist}


When $\mathcal D$ is bounded, of course \ref{D3} is implied by \ref{A8}, but there are cases of interest (as in the kinetic case, see Section~\ref{subsec:kinetic}) where $\mathcal D$ is not bounded.

As we will see (see Lemma~\ref{roglic}), \ref{D4} implies that the (deterministic) interacting process at zero temperature stays in $\mathcal D$, which is clearly a basic requirement to get a Kramers' law. Besides, this condition could be slightly weaken, see Remark~\ref{rem:assumptionInitial}.

Finally, we assume that a Kramers' law holds for the linear system~\eqref{roserouge}. In fact, we need a slightly stronger assumption:

\begin{trivlist}
\item[\customlabel{K}{\A{K}}] 
Let $\tilde X^\sigma$ solve \eqref{roserouge} with $\tilde X_0^\sigma \sim \delta_\lambda$. There exist two families of open domains $\left(\mathcal{D}_{i,\xi}\right)_{\xi>0}$ and $\left(\mathcal{D}_{e,\xi}\right)_{\xi>0}$ with $\mathcal{D}_{i,\xi}\subset\mathcal{D}_{i,0}=\mathcal{D}=\mathcal{D}_{e,0}\subset\mathcal{D}_{e,\xi}$ with the following properties. First, 
\[\sup_{z\in\partial\mathcal{D}_{i,\xi}}{\rm d}\left(z\,;\,\mathcal{D}^c\right)+\sup_{z\in\partial\mathcal{D}_{e,\xi}}{\rm d}\left(z\,;\,\mathcal{D}\right) \ \underset{\xi\rightarrow 0}\longrightarrow  0.\]
Second, for all $\xi>0$ small enough  
 \[\inf_{z\in\partial\mathcal{D}_{i,\xi}}{\rm d}\left(z\,;\,\mathcal{D}^c\right)>\xi \text{ and } \inf_{z\in\partial\mathcal{D}_{e,\xi}}{\rm d}\left(z\,;\,\mathcal{D}\right)>\xi.\]  
 Third, let $\tilde \tau_\sigma(\mathcal D_{u,\xi}) = \inf\{t\geqslant 0,\tilde X_t^\sigma\notin\mathcal D_{u,\xi}\}$ for $\xi\geq0$ and $u\in\{e,i\}$.  For all $\xi\geq0$ and  $u\in\{e,i\}$, there exists $H_{u,\xi}>0$ such that 
it holds for all $\delta>0$:
\begin{equation}
\label{eq:Kramertilde}
\mathbb P \po \exp\po \frac{2}{\sigma^2}\po H_{u,\xi} -\delta \pf\pf  \leqslant \tilde \tau_\sigma(\mathcal D_{u,\xi}) \leqslant \exp\po \frac{2}{\sigma^2}\po H_{u,\xi} +\delta \pf \pf \pf \underset{\sigma\rightarrow0}\longrightarrow 1\,.
\end{equation}
Moreover, for $u\in\{i,e\}$, $\lim_{\xi\to0}H_{u,\xi}=H_{u,0}=:H$.
\end{trivlist}

Since \eqref{roserouge} is a standard (time-homogenous) diffusion process, Assumption~\ref{K} can be checked by applying standard Large Deviation results. We presented it as a black-box assumption for clarity (in particular to avoid a discussion on the characteristic boundary in non-elliptic cases) and refer to Section~\ref{Sec:LDP_linear} for more details.

We can now state our main result.

\begin{thm}\label{th:mr}   Under Assumptions  \ref{A}, \ref{D} and \ref{K}, for all $\delta>0$,
\[\mathbb P \po \exp\po \frac{2}{\sigma^2}\po H  -\delta \pf\pf  \leqslant  \tau_\sigma(\mathcal{D}) \leqslant \exp\po \frac{2}{\sigma^2}\po H  +\delta \pf \pf \pf \underset{\sigma\rightarrow0}\longrightarrow 1\,.\]
\end{thm}

In other words, the Kramers' law holds for the non-linear process \eqref{eq:init}, with the same rate $H$ as the linear process \eqref{roserouge}.

\paragraph{Related works} The existing results devoted to a quantitative measure of the efficiency  of adaptive algorithms are the following: in \cite{LelievreABF} the efficiency is expressed in term of a quantitative long-time convergence speed toward equilibrium (for a process interacting with its law). In \cite{BB,ELM}, for processes interacting with their occupation measures, it is written in term of the asymptotic variance in a CLT. To our knowledge, the only study concerning the exit times of such processes is conducted in \cite{Fortetal}, where the Wang-Landau algorithm is studied for the toy problem of a $3$-states Markov chain (in which case the authors are able to establish that the exit times are sub-exponential in $1/\varepsilon$, which is not the case in our work). 

Besides, concerning more generally the question of exit times for non-linear processes, we already mentioned \cite{EJP,ECP}, where a result similar to Theorem~\ref{th:mr}  is established, but only for the usual elliptic McKean-Vlasov diffusion and in cases where the interaction is convex (and in particular the exit time is \emph{larger} for the interacting process than for the initial dynamics). The general strategy of our proof is in the same spirit as the one of \cite{ECP}.

\paragraph{Organization of the paper} The rest of the paper is organized as follows. We conclude this introduction by a discussion of this result in Section~\ref{subsec:discussion}. Section~\ref{sec:proofs} is devoted to its proof. Examples of applications are provided in Section~\ref{Sec:applications}.

\subsection{First example and discussion}\label{subsec:discussion}

To illustrate Theorem~\ref{th:mr} and fix some ideas, consider the case of the overdamped Langevin process (see Section~\ref{Sec:applications} for other applications). Let $U,W\in\mathcal C^2(\R^d)$ and, for $z\in \R^d$ and $\mu \in \mathcal P_2(\R^d)$,
\[a(z) = -\na U(z) \qquad b(z,\mu) = \int_{\R^d} \na W(z-y) \mu(\dd y)\,.\]
Assume that $U$ is $\rho$-strongly convex for some $\rho>0$, that $W(y)=W(-y)$ for all $y\in \R^d$ and that $\na^2 W$ is bounded. Then $a$ satisfies \eqref{eq:a_contract} (applied with $b=0$)  and, considering $\pi$ a coupling of $\nu$ and $\mu$,
\begin{align*}
|b(z,\nu)-b(y,\mu)| \ &= \ \left|\int_{\R^d\times\R^d} \po \na W(z-v) - \na W(y-w)\pf \pi(\dd v,\dd w)\right| \\
&  \leqslant \ \|\na^2 W\|_{\infty}\po |z-y| +   \int_{\R^d} |v-w|  \pi(\dd v,\dd w) \pf\,. 
\end{align*}
Using the Jensen inequality and taking the infimum over all couplings yields \eqref{pairet} with $\kappa_1=\kappa_2 = \|\na^2 W\|_\infty$. Hence, \ref{A9} holds if $\|\na^2 W\|_\infty < \rho/2$.
 
Denote by $\lambda$ the unique point where $U$ attains its minimum. For $v\in\R^d$,
\[a(\cdot) + b(\cdot,\delta_{v}) = -\na U_{v}\]
with $U_v = U- W(\cdot - v)$. Remark that $\na^2 U_v \geqslant \rho/2>0$, so that $U_v$ is convex, and moreover $\na U_\lambda(\lambda) = 0$. It means that, in Theorem~\ref{th:mr}, the linear process \eqref{roserouge} reads
\[\dd Z_t \ = \ -\na U_\lambda(Z_t) + \sqrt{2}\sigma \dd B_t\,.\]
Take $\mathcal D=\mathbb B(\lambda,r)$ for some $r>0$, so that \ref{D} holds for all initial distribution with support included in $\mathcal D$. Setting $\mathcal D_{e,\xi} = \mathbb B(\lambda,r+\xi)$ and $\mathcal D_{i,\xi}=\mathbb B(\lambda,r-\xi)$ for $\xi<r$, \ref{K} holds with 
\[H_{u,\xi} = \inf_{x\in\partial\mathcal D_{u,\xi}} U_{\lambda}(x)-U(\lambda)\,,\qquad H=\inf_{x\in\partial\mathcal D} U_{\lambda}(x)-U(\lambda)\,,\]
see Section~\ref{Sec:LDP_linear}.
Similarly, the rate for the initial linear process~\eqref{linear_convex} is 
\[L=\inf_{x\in\partial\mathcal D} U(x)-U(\lambda)\,.\]
As a consequence, as announced, our assumptions allow for situations where $H<L$, which is here the case as soon as $\inf_{x\in\partial \mathcal D} W(x-\lambda)>0$.

For instance, if $U(z) = \rho/2|z-\lambda|^2$ and $W(z)=\alpha|z|^2$, then  \ref{A9} holds whenever $\alpha <\rho/4$, and $U_\lambda = (\rho/2-\alpha)|z-\lambda|^2$.  In that case, $L=\rho/2 r^2$ and $H=(\rho/2-\alpha) r^2$, which can be made arbitrarily close to $L/2$ by taking $\alpha$ arbitrarily close to $\rho/4$.

\begin{rem}
\label{AG2R}
More generally, we do not need to assume that $z\mapsto b(z,\delta_\lambda)$ derives from a potential. Considering the variational definition of $H$ (see   Section~\ref{Sec:LDP_linear}), we can see  that the hypothesis $\left\langle b(z,\delta_\lambda);b(z,\delta_\lambda)-4\nabla U(z)\right\rangle<0$ for any $z\in\bRb^d$ is sufficient to ensure that the exit-cost is reduced. The idea is to consider the optimal trajectory for the linear process \eqref{linear_convex} and to show that, with the interaction, the cost of this trajectory gets strictly smaller than $L$.
\end{rem}

This example calls for a few remarks, in particular in view of the initial algorithmic motivations discussed in Section~\ref{sec:motivations}. Here we assumed that $U$ is convex, which means that in fact the non-perturbed overdamped Langevin process is not metastable and the importance sampling scheme presented in  Section~\ref{sec:motivations} is not relevant. More generally, we see that the global contraction property \ref{A9} is very strong and rules out any interesting practical case for the adaptive algorithms. In fact, Theorem~\ref{th:mr} has to be understood as a first step toward proving a similar result where the convexity of $U$ is only assumed locally on a neighborhood of $\lambda$. Indeed, if the initial distribution of the process is supported on a neighborhood of $\lambda$ then the probability that the process exits in a time smaller than   $e^{2(H-\delta)/\sigma^2}$ for some $\delta>0$ will be small (and so will be the mass of $m_t^\sigma$ far from $\lambda$) which means that only the local behavior of $a$ is relevant up to these times.

In other words, with Theorem~\ref{th:mr}, we do not really prove that the metastability is reduced for a multi-modal probability target, which is the final goal. We only prove that, for a process attracted to a single point, the time to exit from a ball centered at the attractor can be reduced by adding a repelling interaction potential. This is still a new and far from trivial result and, as mentioned above, it should be possible to combine it with a localization procedure in order to treat a non-convex case. This will be the topic of a future work. More precisely, we expect the localization argument to enable the study of the case where $U$ is not convex but  $\mathcal D$ is still a domain on which $U$ is convex. Going beyond this assumption will raise additional difficulties.

Let us make another remark on the local character of our result. In Theorem~\ref{th:mr}, the initial condition   $m_0$ is   concentrated on the domain $\mathcal D$, which in the example of this section is a ball centered at the minimum of $U$. Now, consider the practical situation of  a system of  $N$ interacting particles in a non-convex potential $U$. It is possible that, at some point, half the particles are in some well of $U$, and the other half is in another one. This is precisely why it is more natural to use an adaptive algorithm rather than a  biasing potential $W(\cdot-\lambda)$ for a fixed $\lambda$. Using an interaction potential $W(y)=\alpha|y|^2$ would have the effect that the particles in the first well are repelled by the particles in the second one, even if the two wells are far away one from the other. Yet, there is no particular reason to believe that an efficient way to escape from the first well is to go in the direction opposite to the second one. In fact, in practice, localized repelling interactions are used, like $W(y)=\alpha\exp(-\beta|y|^2)$, as in the metadynamics algorithm \cite{Metadyn1,Metadyn2,Metadyn3}. It means that the effect of the particles in a given well on those in another well is negligible. In that case, Theorem~\ref{th:mr} can be understood as a simplification of the problem, where each well is treated separately (which is reasonable and could possibly be made rigorous for domains $\mathcal D$ such that the exit-time is small with respect to the transition time from one well to another). 
 A statement about a general situation with a non-localized interaction would be harder to interpret that the simple situation of a single well as in Theorem~\ref{th:mr}, which is why we focus on the latter in this work. 
 
 To be more precise, notice that, in a convex potential, even if a particle exits at some point from a ball centered at the stable point, it will fall back shortly after, since there is nowhere else to go. The mass around the stable point is thus constant. This is not at all the case if a particle which exits $\mathcal D$ falls in the basin of attraction of a different stable point: there is a mass leakage, which may diminish the strength of the repulsion, and thus slow down the exit of the remaining particles. There may be cases where the proof of Theorem~\ref{th:mr} still works and yields a Kramers'law similar to the linear process with interaction $b(\cdot, p\delta_{\lambda_1}+(1-p)\delta_{\lambda_2})$ where $\lambda_1$ and $\lambda_2$ are the local minima in each well and $p$ is the initial proportion of particles in the first well, but such a result can only hold if the transition from one well to another happens at a time much larger than the exit of $\mathcal D$ (which can be for instance the union of two balls centered at $\lambda_1$ and $\lambda_2$), in which case all particles will typically exit $\mathcal D$ before any transition from one well to the other is observed. From the point of view of numerical acceleration (where one indeed wants to accelerate the transitions), this is still not the interesting framework. However, our strategy of proof could lead, in the case where exits of $\mathcal D$ correspond to transitions between different metastable states, to the following weaker result:
 \[\forall \delta>0\,, \qquad \liminf_{\sigma\rightarrow0}\mathbb P\po   \tau_\sigma(\mathcal{D}) < e^{(H+\delta)/\sigma^2}\pf \ > \ 0\,,\]
 where $H$ is the rate of an explicit linear process (depending on the initial law).

 For instance, let us consider a toy problem that gives a simplified picture of a system in $\R^d$ with two minima where the first particles to cross slow down the others, which gives a behaviour different than a simple Kramers' law with a modified height. Consider the Markov chain on $\{0,1\}$ with rates $\lambda(i\rightarrow j)=e^{-a_{ij}/\sigma^2}$ for some $a_{01},a_{10}>0$, with $\mathcal D= \{0\}$ and the initial distribution $m_0=\delta_0$. Then $\tau_\sigma(\mathcal{D})$ is a geometric law with parameter $e^{-a_{01}/\sigma^2}$. Now, add some self-repulsion by considering the chain with rates $\lambda(i\rightarrow j)=e^{-b_{ij}(m_t^\sigma)/\sigma^2}$ where $b_{ij}(\nu)=a_{ij}+\alpha (\nu(j)-\nu(i))$ for some $\alpha>0$. Then the law of $\tau_\sigma(\mathcal{D})$ is given by
\[\mathbb P\po \tau_\sigma(\mathcal{D}) > t \pf \ = \ \exp \po - \int_0^t e^{(\alpha(2x_s-1)-a_{01})/\sigma^2} \dd s\pf \]
where $x_t = m_t^\sigma(0)$ solves
\[\dot x_t \ = \ - e^{(\alpha(2x_t-1)-a_{10})/\sigma^2} x_t + e^{(\alpha(1-2x_t)-a_{01})/\sigma^2}(1-x_t)\,.\]
Suppose to fix ideas that $a_{01}=a_{10}=a$, so that $x_t \rightarrow 1/2$ as $t\rightarrow +\infty$ and $x_t\geqslant 1/2$ for all $t\geqslant 0$. Suppose also that $\alpha<a$. Then it is not difficult to see  that for all $H\in [a-\alpha,a]$ and all $\delta>0$,
\[\liminf_{\sigma\rightarrow0}\mathbb P\po e^{(H-\delta)/\sigma^2} < \tau_\sigma(\mathcal{D}) < e^{(H+\delta)/\sigma^2}\pf \ > \ 0\,.\]

\section{Proofs} 
 \label{sec:proofs}
 
In order to highlight the main steps of  the proof of Theorem~\ref{th:mr}, this section is organised as follows. First, we present in Section~\ref{sec:couplings} a general result based on the parallel coupling of two diffusion processes, which will be  intensively used in all the rest of the proof. The existence and uniqueness of the process \eqref{eq:init} is addressed in Section~\ref{sec:existenceprocess}. The long-time convergence toward equilibrium (at a speed that does not depend on the temperature) is established in Section~\ref{sec:longtime}, while the low noise asymptotics (both at equilibrium and in finite time intervals) is analysed in Section~\ref{sec:lownoise}. Finally, building upon all these intermediary results, the proof of Theorem~\ref{th:mr} is given in Section~\ref{sec:finalproof}.

\subsection{Parallel coupling}\label{sec:couplings}

We consider in this section a general framework. First, let $a\in\mathcal C^1(\R^d,\R^d)$,  $b\in \mathcal C^0(\R^d\times \mathcal P_2(\R^d),\R^d)$, $M\in\R^{d\otimes d}$ and $(B_t)_{t\geqslant 0}$ be a standard Brownian motion on $\R^d$. In all Section~\ref{sec:couplings}, these parameters are fixed.

Second, let $\sigma \geqslant 0$,  $m_0 \in \mathcal P_2(\R^d)$, $Y_0 \sim m_0$ and let $\nu = (\nu_t)_{t\geqslant 0}\in \mathcal C^0(\R_+,\mathcal P_2(\R^d))$. We say that $(Z_t)_{t\geqslant 0}$ is a process associated to $Y_0$, $\sigma$ and $\nu$ if, almost surely, for all $t\geqslant 0$,
\begin{equation}\label{eq:Zgeneral}
Z_t = Y_0 + \sigma M B_t + \int_0^t  \po a(Z_s) + b(Z_s,\nu_s) \pf \dd s.
\end{equation}

\begin{prop}
\label{prop:parallelcoupling}
Assume \ref{A8} and \ref{A9}. Let $\sigma,\tilde\sigma\geqslant 0$, $m_0,\tilde m_0\in \mathcal P_2(\R^d)$, $Y_0\sim m_0$, $\tilde Y_0\sim \tilde m_0$, $\nu,\tilde \nu\in \mathcal C^0(\R_+,\mathcal P_2(\R^d))$.
Let $(Z_t)_{t\geqslant 0}$ and $(\tilde Z_t)_{t\geqslant 0}$ be  processes associated respectively to $\sigma,Y_0,\nu$ and $\tilde \sigma,\tilde Y_0,\tilde\nu$, in the sense of \eqref{eq:Zgeneral}. Set $\Delta Z_t = Z_t - \tilde Z_t$ and $f(t) = \mathbb E \po |\Delta Z_t|^2 \pf$.
\begin{enumerate}
\item If $\sigma = \tilde \sigma$, then almost surely, $t\mapsto |\Delta Z_t|^2$ is $\mathcal C^1$ with for all $t \geqslant 0$,
\[\partial_t |\Delta Z_t|^2 \ \leqslant \   -  2\rho   | \Delta Z_t|^2 +2  \kappa\mathbb W_2^2(\nu_t,\tilde\nu_t)\,.\]
\item If $\tilde \sigma = 0$ then $f$ is $\mathcal C^1$ with for all $t\geqslant 0$,  
\[f'(t)  \leqslant d \sigma^2   \|M\|^2  -  2\rho  f(t) +   2\kappa \mathbb W_2^2(\nu_t,\tilde\nu_t)\,.\]
\end{enumerate}
\end{prop}

\begin{proof}
If $\sigma = \tilde\sigma$, 
\[\dd \Delta Z_t \ = \ \po a(Z_t) + b(Z_t,\nu_t)- a(\tilde Z_t)-b(\tilde Z_t,\tilde \nu_t)\pf \dd t \,.\]
The right-hand side is continuous, and \ref{A9} yields
\begin{eqnarray*}
 \dd   |\Delta Z_t|^2& = & 2 \Delta Z_t \cdot  \po a(Z_t) + b(Z_t,\nu_t)- a(\tilde Z_t)-b(\tilde Z_t,\tilde \nu_t)\pf\dd t\\
& \leqslant & - 2 \rho | \Delta Z_t|^2\dd t   + 2 \kappa \mathbb W_2^2(\nu_t,\tilde \nu_t)\dd t .
\end{eqnarray*}
Using Ito's formula, the second point follows the same line.
\end{proof}

Let us now describe a first consequence  of this general result in the deterministic case where $\sigma=\tilde\sigma =0$ and $\nu$ and $\tilde \nu$ are constant Dirac masses.

\begin{lem}
\label{cor:deterministic_flow}
Under \ref{A},
\begin{enumerate}
\item for all $v \in \R^d$ and $z\in\R^d$, the solution of   the Cauchy problem $\dot z_t = a(z_t)+b(z_t,\delta_{v})$ with $z_0=z$ is defined for all positive times. We denote    by $(\psi_t^v)_{t\geqslant 0}$ the associated flow (so that $z_t = \psi_t^v(z)$).
\item for all  $v\in \R^d$, $z,y\in \R^d$ and $t\geqslant 0$,
\[|\psi_t^v(z)-\psi_t^v(y)| \ \leqslant\   e^{-\rho t}|z-y|\,.\]
\item there exists a unique $\lambda \in \R^d$ such that $\psi_t^\lambda(\lambda)=\lambda$ for all $t\geqslant 0$.
\end{enumerate}
\end{lem}

\begin{proof}
Using \ref{A9}, for $v\in\R^d$, a solution of $\dot z_t = a(z_t)+b(z_t,\delta_v)$ is such that
\[\partial \po |z_t|^2\pf \leqslant -2\rho  |z_t|^2 + 2 |z_t| \po |a(0)| + |b(0,\delta_v)|\pf \,,\] 
which by Gr{\"o}nwall's Lemma implies non-explosion, hence the flow is defined for all positive times.

Let $z,y,v\in\R^d$. Consider the settings of Proposition~\ref{prop:parallelcoupling} with $\sigma=\tilde\sigma=0$,  $\nu_t=\tilde\nu_t = \delta_v$ for all $t\geqslant 0$, $Y_0=z$ and $\tilde Y_0 = y$. Write $g(t) = |\psi_t^v (z) - \psi_t^v(y)|^2 $ for $t\geqslant 0$.
Using that $\mathbb W_2(\nu_s,\tilde\nu_s)=0$ for all $s\geqslant 0$ in that case, we immediately get from the first point of Proposition~\ref{prop:parallelcoupling}  that $g'(t) \leqslant - 2\rho  g(t)$ for all $t\geqslant 0$, which proves the second point.  Moreover, this implies that, for all $v\in \R^d$, the flow $\psi^v$    admits a unique equilibrium, that we denote $\Pi(v)$.
 
 Fix $\lambda_1,\lambda_2 \in \R^d$ and set $Z_t = \Pi(\lambda_1)$ and $\tilde Z_t=\Pi(\lambda_2)$ for all $t\geqslant 0$. Then $Z$ (resp. $\tilde Z$) solves \eqref{eq:Zgeneral} with $\sigma=0$, $Y_0 = \Pi(\lambda_1)$  (resp. $\Pi(\lambda_2)$) and $\nu_t =\delta_{\lambda_1}$ (resp. $\delta_{\lambda_2}$)  for all $t\geqslant0$. The first point of Proposition~\ref{prop:parallelcoupling} then reads
\[ |\Pi(\lambda_1)-\Pi(\lambda_2)|^2 \ \leqslant \ \frac{\kappa}{\rho} |\lambda_1-\lambda_2|^2 \,. \]
The condition $\rho>\kappa$  implies that $\Pi$ is a contraction. As a consequence, it admits a unique fixed point $\lambda$. 
\end{proof}
 
\subsection{Existence of the process}\label{sec:existenceprocess}

This section is devoted to the proof of Proposition~\ref{prop:gen:wp}. Before entering the proof, let us state the following crucial Lemma.
\begin{lem}\label{lem:W2noyau}
For  $t\geqslant0$, let $R(t,\cdot)$ be a probability measure on $[0,t]$ and,  for $i=1,2$, let $(m_{s,i})_{s\in[0,t]}$ be a family of probability measures in $\mathcal P_2(\R^d)$ and $\mu_{t,i} = \int_0^t m_{s,i}  R(t,\dd s)$. Then
\[\mathbb W_2^2\po \mu_{t,1},\mu_{t,2}\pf \leqslant \int_0^t \mathbb W_2^2\po m_{s,1},m_{s,2}\pf  R(t,\dd s)\]
\end{lem}
\begin{proof}
Let $S$ be a random variable distributed according to $R(t,\cdot)$. Conditionally to $S$, let $(Y_1,Y_2)$ be an optimal $\mathbb W_2$-coupling of $m_{S,1}$ and $m_{S,2}$. Then $(Y_1,Y_2)$ is a coupling of $\mu_{t,1}$ and $\mu_{t,2}$, so that
\[\mathbb W_2^2\po \mu_{t,1},\mu_{t,2}\pf \  \leqslant \ \mathbb E \po |Y_{1}-Y_2|^2\pf \ = \ \int_0^t \mathbb W_2^2 \po m_{s,1},m_{s,2}\pf R(t,\dd  s)\,.\]
\end{proof}

\begin{proof}[Proof of Proposition \ref{prop:gen:wp}]
Let $T<\infty$.
Denote first  a linearized version of the process \eqref{eq:init}, where the dependence of the law in the drift is replaced by  the family  $\bar{m}\in \mathcal C([0,T], \mathcal P_2(\R^d))$, by $\bar{X}$: 
\begin{equation}
\label{eq:linearisedX}
 d\bar{X}_t= a(\bar{X}_t) \dd t+ \bar{b}(t, \bar{X}_t)\dd t + \sigma M \dd B_t   
\end{equation}
where for $(t,x) \in \R_+ \times \R^d$ we set
$$\bar{b}(t, x)= b\po x, \int_0^t \bar{m}_s R(t,ds)\pf.$$

From \ref{A8} and \ref{A9} one easily has for any $t\leq T$ and $x\in\R^d$
$$ x \cdot (a(x) + \bar{b}(t,x)) \leq -\rho |x|^2 + x \cdot (a(0)+ \bar{b}(t,0)) .$$
Hence, on the  one hand,  if $|x|$ is large enough then
$$x \cdot (a(x) + \bar{b}(t,x)) \leq 0.$$
On the other hand,  since,  by \ref{A},
$$|\bar{b} (t,0)|\leq |b (0,\delta_0)|+ C \mathbb{W}_2\po \int_0^t \bar{m}_s R(t,ds)), \delta_0 \pf   \leq |b (0,\delta_0)|+C  \sup_{s\leq T }  \po \int_{\R^d} |y|^2 \bar{m}_s (dy) \pf^{1/2}  $$
for some $C>0$, we get that  for $|x|$ small,
$$ x \cdot 
(a(x) + \bar{b}(t,x)) \leq  |x|  \left(|a(0)|+ |b (0,\delta_0)|+C  \sup_{s\leq T }  \po \int_{\R^d} |y|^2 \bar{m}_s (dy) \pf^{1/2}\right). $$

Thus, we can apply \cite[Thm. 10.2.2, p. 255]{SV} and conclude the existence of a weak solution to the SDE in \eqref{eq:linearisedX}. As the coefficients in \eqref{eq:linearisedX} are locally Lipschitz continuous by \ref{A}, this SDE admits strong uniqueness. Hence, by Yamada and Watanabe principle, \eqref{eq:linearisedX}  admits a unique strong solution up to any time horizon $T>0$.

Now,  for $m\in \mathcal C([0,T], \mathcal P_2(\R^d))$ define $\Phi$  by $\Phi(m) = \big(\mathcal L(X_t^{\sigma,m})\big)_{t \in [0,T]}$, where $X^{\sigma,m}$ denotes the solution of \eqref{eq:init} on $[0,T]$ parametrized by $\int_0^t m_s dR_t(s)$. Note that,  from the computations done in the proof of Proposition \ref{prop:parallelcoupling}, $\Phi$ maps  $\mathcal C([0,T], \mathcal P_2(\R^d))$ onto itself. Still from the quoted proof, one gets that 
$\Phi$ is a contraction so that it admits a unique fixed point $m$ that satisfies that the  family of one-dimensional time marginal laws of the process $X^{\sigma,m}$ are $m$.  Hence, we constructed a solution the non linear SDE \eqref{eq:init}. This concludes the proof on $[0,T]$ for any $T>0$ and so on $\R^+$.
\end{proof}

\subsection{Long-time behavior}\label{sec:longtime}

We start by proving a Gr\"onwall-type lemma in the presence of a memory kernel $R$.

\begin{lem}
\label{lem:Gronwall_memory}
Let $R$ satisfy \ref{A4} and $\alpha,\beta\geqslant0$ with $\alpha>\beta$. There exists a decreasing $x:\R_+\rightarrow \R_+$ with $x(t)\rightarrow 0$ as $t\rightarrow +\infty$ such that, for all $f\in\mathcal C^1(\R_+,\R_+)$ and $\gamma\geqslant0$ such that for all $t\geqslant0,$
\[f'(t) \ \leqslant \ - \alpha f(t)+\beta\int_0^tf(s)R(t,\dd s) + \gamma\,,\]
 then for all $t\geqslant 0$
 \[f(t) \ \leqslant \ \frac{\gamma}{\alpha-\beta} + x(t) \po f(0) - \frac{\gamma}{\alpha-\beta}\pf_+\,.\]
\end{lem}

\begin{proof}
Set $a=\gamma/(\alpha-\beta)$, and let $b>a$. As a first step, let us prove that, if $f(0)\leqslant b$, then $f(t)\leqslant b$ for all $t\geqslant 0$. Suppose that $s=\inf\{t\geqslant 0, f(t)>b\}$ is finite. Then $f(s)=b$ and $\int_0^s f(u)R(s,\dd  u) \leqslant b$, and thus
\[f'(s) \ \leqslant \  (- \alpha+\beta)(b-a) \ < \ 0\,,\]
which yields a contradiction, and prove the claim. Since we can take any arbitrary $b>a$, we obtain that, if $f(0) \leqslant a$, then $f(t) \leqslant a$ for all $t\geqslant 0$, which concludes the proof of the lemma in this case. 

For the rest of the proof, we assume that $f(0)>a$. The previous result applied with $b=f(0)$ shows that $f(t) \leqslant f(0)$. Let $g(t) = (f(t)-a)/(f(0)-a)$, which is such that, for all $t\geqslant 0$, $ g(t) \leqslant g(0)=1$  and
\[g'(t) \ \leqslant \ -\alpha g(t)  +\beta\int_0^tg(s)R(t,\dd s)\,. \]
In particular, $g'\leqslant-\alpha g + \beta$, and thus $g \leqslant x_0$ where, for all $t\geqslant 0$,
\[x_0(t) \ = \ e^{-\alpha t} + \frac{\beta}{\alpha}(1-e^{-\alpha t})\,. \]
Let $c=\sqrt{\beta/\alpha} $, so that $\beta/\alpha < c < 1$, and $t_0=0$. Suppose by induction that a function $x_n:\R_+\rightarrow \R_+$ and a time $t_n\geqslant n$  have been defined for some $n\in\N$, with $x_n(t)\rightarrow c^{n+2}$ as $t\rightarrow+\infty$. Let $t_{n+1}\geqslant 1+t_n$ be large enough so that
\[x_n(t_{n+1}) \leqslant c^{n+1}\qquad\text{and}\qquad \int_0^t x_n(s)R(t,\dd s) \leqslant c^{n+1}\ \forall t\geqslant t_{n+1}\,,\]
which is possible thanks to \ref{A4} and the fact $c^{n+1}>c^{n+2}$. Define $x_{n+1}(t) = x_n(t)$ for $t< t_{n+1}$ and
\[x_{n+1}(t) = e^{-\alpha (t-t_{n+1}) }c^{n+1} + \po 1 - e^{-\alpha (t-t_{n+1}) }\pf c^{n+3}  \]
for $t\geqslant t_{n+1}$. Then, $x_{n+1}(t)\rightarrow c^{n+3}$ as $t\rightarrow +\infty$, which concludes the definition by induction of $x_n$ and $t_n$ for all $n\in \N$. Remark that this construction only involves $\alpha,\beta,R$. Let us prove that $g\leqslant x_n$ for all $n\in \N$. We have already treated the case $n=0$, suppose that this is true for some $n\in\N$. For $t<t_{n+1}$, $g(t)\leqslant x_n(t)=x_{n+1}(t)$. For $t\geqslant t_{n+1}$,
\[g'(t) \ \leqslant \ -\alpha g(t)  +\beta\int_0^tg(s)R(t,\dd s)
\ \leqslant \ -\alpha g(t)  +\beta\int_0^tx_n(s)R(t,\dd s)
\ \leqslant \ -\alpha g(t)  +\beta c^{n+1}\,, \]
and thus, using that $g(t_{n+1})\leqslant x_n(t_{n+1})\leqslant c^{n+1}$,
\[g(t) \ \leqslant 
e^{-\alpha (t-t_n) }g(t_n) + \po 1 - e^{-\alpha (t-t_n) }\pf c^{n+3} \ \leqslant  \ x_{n+1}(t)\,,\]
which concludes the proof that $g\leqslant x_n$ for all $n\in\N$. Define $x(t) = c^{n}$ for $t\in[t_n,t_{n+1})$ (remark that $t_n\geqslant n \rightarrow +\infty$, so that $x(t)$ is indeed defined for all $t\geqslant 0$). Then $g\leqslant x$, which concludes the proof of the lemma.
\end{proof}

\begin{rem}
In the case where $R_t=\delta_t$ for all $t\geqslant 0$, of course the result holds with $x(t) = e^{-(\alpha-\beta)t}$. In the uniform case, assuming that $f'=-\alpha f + \beta F$ with $F(t) = t^{-1}\int_0^t f(s) \dd s$, we see that, for large $t$, $F(t)$ evolves slowly, so the evolution of $f$ is approximately $f'(t) \simeq  - \alpha f(t) + \beta F(t_0)$ for  with $0\leqslant t-t_0 \ll t_0$, so that $f(t) \simeq \beta/\alpha F(t)$ and $F'(t) \simeq -(1-\beta/\alpha)/t F(t)$. As a consequence, $F$ (hence $f$) goes to $0$ as $t^{\beta/\alpha -1}$.
\end{rem}

\begin{prop}
\label{prop:longtime}
Under \ref{A}, there exists a positive  function $Q$ on $\R_+$ that depends only on $\rho$ and $\kappa$ and vanishes at infinity such that the following holds. Let $(X_t^\sigma,\mu_t^\sigma,m_t^\sigma)_{t\geqslant 0}$ and $(\tilde X_t^\sigma,\tilde \mu_t^\sigma,\tilde m_t^\sigma)_{t\geqslant 0}$  solve \eqref{eq:init} with respective initial distributions $m_0^\sigma,\tilde m_0^\sigma\in\mathcal P_2(\R^d)$. Then for all $t\geqslant 0$,
\[\mathbb W_2 (m_t^\sigma,\tilde m_t^\sigma) + \mathbb W_2 (\mu_t^\sigma,\tilde \mu_t^\sigma) \ \leqslant \ Q(t)\mathbb W_2 (m_0^\sigma,\tilde m_0^\sigma)\,.  \]
\end{prop}
\begin{proof}
We can consider a copy of the processes since the statement only concerns the distributions. We consider  $(Y_0,\tilde Y_0)$ an optimal $\mathbb W_2$ coupling of $m_0^\sigma$ and $\tilde m_0^\sigma$. Then $(X_t^\sigma)_{t\geqslant 0}$ (resp. $(\tilde X_t^\sigma)_{t\geqslant 0}$) has the same law as the solution $(Z_t)_{t\geqslant 0}$ (resp. $(\tilde Z_t)_{t\geqslant 0}$) of \eqref{eq:Zgeneral} associated to $Y_0$, $\sigma$ and $\nu =(\mu_t^\sigma)_{t\geqslant 0}$ (resp. $\tilde Y_0$, $\sigma$ and $(\tilde \mu_t^\sigma)_{t\geqslant 0}$). In particular, $(Z_t,\tilde Z_t)$ is a coupling of $m_t^\sigma$ and $\tilde m_t^\sigma$ for all $t\geqslant 0$ which, together with Lemma~\ref{lem:W2noyau}, implies that
\[\mathbb W_2^2(\mu_t^\sigma,\tilde \mu_t^\sigma) \ \leqslant \  \int_0^t f(s) R(t,\dd s)\]
with $f(s) = \mathbb E (| Z_s - \tilde Z_s|^2)$ for all $s\geqslant 0$. Taking the expectation in the first point of Proposition~\ref{prop:parallelcoupling}, we get that, for all $t\geqslant 0$,
\begin{equation}\label{eq:f'}
f'(t) \ \leqslant \   -2\rho f(t) +2\kappa \int_0^t f(s) R(t, \dd s) .
\end{equation}
 The conclusion follows, thanks to \ref{A4}, by applying Lemmas~\ref{lem:Gronwall_memory} and \ref{lem:W2noyau} to get
\[\mathbb W_2 (m_t^\sigma,\tilde m_t^\sigma) + \mathbb W_2 (\mu_t^\sigma,\tilde \mu_t^\sigma) \ \leqslant \ f(t) + \int_0^t f(s) R(t,\dd s) \ \leqslant \ \co x(t) + \int_0^t x(s) R(t,\dd s)\cf f(0) \,. \]
\end{proof}

\begin{cor}\label{cor:equilibrium}
Under \ref{A}, there exists a unique $m_\infty^\sigma \in \mathcal P_2(\R^d)$ which is stationary for \eqref{eq:init}, in the sense that the process $(Z_t)_{t\geqslant 0}$ solving \eqref{eq:Zgeneral} with $Y_0 \sim m_\infty^\sigma$ and $\nu_t = m^\sigma_\infty$ for all $t\geqslant0$ is such that $Z_t \sim m_\infty^\sigma$ for all $t\geqslant 0$.
\end{cor}
\begin{proof}
Remark that the fact that $\mu \in \mathcal P_2(\R^d)$ is a fixed point for \eqref{eq:init} does not depend on $R$. As a consequence, we only consider the classical non-linear McKean-Vlasov case, i.e. $R(t,\cdot)=\delta_t$ for all $t\geqslant 0$.  For $s\geqslant 0$, let $\Phi_s:\mathcal P_2(\R^d)\rightarrow \mathcal P_2(\R^d)$ be defined by $\Phi_s(m_0^\sigma)=m_s^\sigma$. In the classical non-linear case, we have $\Phi_s(m_t^\sigma)=m_{s+t}^\sigma$ for all $t,s\geqslant 0$. Let $s\geqslant 0$ be large enough so that, by Proposition~\ref{prop:longtime}, $\Phi_s$ is a contraction of $(\mathcal P_2(\R^d),\mathbb W_2)$, which is a complete metric space. Denote $m_\infty^\sigma$ the unique fixed point of $\Phi_s$. Using that $\Phi_{t+s} = \Phi_t \circ \Phi_s=\Phi_s \circ \Phi_t$ for all $t\geqslant 0$, we get that $\Phi_t(m_\infty^\sigma)$ is a fixed point of $\Phi_s$ for all $t\geqslant 0$ which, by uniqueness, implies that $\Phi_t(m_\infty^\sigma)=m_\infty^\sigma$ for all $t\geqslant 0$.
\end{proof}

\subsection{Low noise asymptotics}\label{sec:lownoise}

First, we state the low noise convergence of the stationary distribution.

\begin{prop}\label{cor:equilibrium_lownoise}
Under \ref{A}, let $\lambda\in \R^d$ be given by Lemma~\ref{lem:exist_lambda} and, for $\sigma \geqslant 0$, let $m_\infty^\sigma$ be given by Corollary~\ref{cor:equilibrium}. For all $\sigma\geqslant 0$,
\[\mathbb W_2^2\po m_\infty^\sigma,\delta_\lambda\pf \ \leqslant \ \frac{d \|M\|^2 }{2(\rho-\kappa)} \sigma^2\,.\]
\end{prop}
\begin{proof}
Let $Z$ (resp. $\tilde Z$) solve \eqref{eq:Zgeneral}  with $Y_0\sim m_\infty^\sigma$, $\sigma$ and $\nu_t = m_\infty^\sigma$ for all $t\geqslant 0$ (resp. $\tilde Y_0=\lambda$, $\tilde \sigma=0$ and $\nu_t = \delta_{\lambda}$ for all $t\geqslant 0$). In particular, for all for all $t\geqslant 0$, $Z_t\sim m_\infty^\sigma$,  $\tilde Z_t = \lambda$ and $\mathbb W_2^2\po m_\infty^\sigma,\delta_\lambda\pf = \mathbb E ( |Z_t-\lambda|^2)$. The result then is a straightforward consequence of the second point of Proposition~\ref{prop:parallelcoupling}.
\end{proof}

Second, we consider the low noise convergence of the process on finite time intervals.

\begin{prop}\label{prop:small_time_coupling}
Assume \ref{A}, let $m_0\in\mathcal P_2(\R^d)$ and $Y_0\sim m_0$. There exists $\mathcal K\in \mathcal C^0(\R_+,\R_+)$ such that the following holds. For a fixed Brownian motion $(B_t)_{t\geqslant 0}$, for all $\sigma\geqslant 0$, let $(X_t^\sigma)_{t\geqslant 0}$ be the solution of  \eqref{eq:init} with $X_0^\sigma =Y_0$. For all $t\geqslant 0$, $\sigma\geqslant 0$ and $\varepsilon\in(0,1]$,
\begin{equation}\label{eq:tempscourt}
\mathbb P \po \sup_{s\in[0,t]}|X_s^{\sigma} - X_s^{0}| >\varepsilon \pf \ \leqslant \     \frac{\sigma \mathcal K(t) }{\varepsilon} \,.
\end{equation}
\end{prop}

\begin{proof}
Let $f(t) = \mathbb E(|X_s^\sigma-X_s^0|^2)$ for $t\geqslant 0$. In particular, $ \mathbb W_2^2(m_t^\sigma,m_t^0)\leqslant   f(t)$ for all $t\geqslant 0$ and, applying Lemma~\ref{lem:W2noyau} and the second point of Proposition~\ref{prop:parallelcoupling}, 
\[f'(t) \ \leqslant \ \gamma - 2\rho f(t) + 2\kappa \int_0^t f(s)R(t,\dd s)\]
with $\gamma=d\sigma^2   \|M\|^2$. Since $f(0)=0$, Lemma~\ref{lem:Gronwall_memory} implies that $f(t) \leqslant \gamma/(2(\rho-\kappa))$ for all $t\geqslant 0$, and thus
\[\mathbb W_2^2(m_t^\sigma,m_t^0) \ \leqslant \ \frac{d \|M\|^2\sigma^2}{2(\rho-\kappa)}\]
for all $t\geqslant 0$, which also implies 
\[\mathbb W_2^2(\mu_t^\sigma,\mu_t^0) \ \leqslant \ \frac{d \|M\|^2\sigma^2}{2(\rho-\kappa)}\]
for all $t\geqslant 0$. 

Recall that, from \ref{A8} and  \ref{D3}, $a$ and $b$ are Lipschitz continuous on $\mathcal D_{+1}$. Moreover, \ref{D4} implies that $X_t^0\in\mathcal D$ for all $t\geqslant 0$, see Lemma~\ref{roglic}. Let $\tau_{+1} = \inf\{t\geqslant 0, X^\sigma_t\notin \mathcal D_{+1}\}$. Then, for some $L>0$, for all $t\leqslant \tau_{+1}$,
\begin{align*}
 |X_t^\sigma - X_t^0| \ & \leqslant \ \sigma \|M\| |B_t| + L \int_0^t \po   |X_s^\sigma - X_s^0| + \mathbb W_2 \po \mu_s^\sigma,\nu_s^0 \pf \pf \dd s \,.
\end{align*}
Gr{\"o}nwall's Lemma yields
\[\text{almost surely, }\forall t\in [0,\tau_{+1}],\qquad \sup_{s\in[0,t]}|X_s^\sigma - X_s^0| \ \leqslant \ \sigma e^{Lt}\po \|M\| \sup_{s\in[0,t]}|B_s| + C t\pf\,, \]
with $C=L \|M\|\sqrt{d/(2(\rho-\kappa))}$.

Now, fix $t\geqslant 0$ and $\varepsilon\in(0,1]$. Remark that $|X_s^\sigma-X_s^0| \leqslant \varepsilon$ for some $s\geqslant 0$ implies that $d(X_s^\sigma,\mathcal D)\leqslant 1$. Hence,
\[\left\{\sigma e^{Lt}\po \|M\| \sup_{s\in[0,t]}|B_s| + C t\pf \leqslant \varepsilon\right\} \subset \left(\{\tau_{+1} \geqslant t\} \cap \left\{\sup_{s\in[0,t]}|X_s^\sigma - X_s^0| \leqslant \varepsilon \right\}\right)\,.\]
As a conclusion, if $\sigma e^{Lt} C t >\varepsilon/2$ we simply bound
\[\mathbb P \po \sup_{s\in[0,t]}|X_s^{\sigma} - X_s^{0}| >\varepsilon \pf \ \leqslant \ 1 \ \leqslant \ \frac{2\sigma e^{Lt} C t}{\varepsilon} \,,\]
while, if $\sigma e^{Lt} C t \leqslant \varepsilon/2$, we can bound
\[\mathbb P \po \sup_{s\in[0,t]}|X_s^{\sigma} - X_s^{0}| >\varepsilon \pf \ \leqslant \ \mathbb P \po \sup_{s\in[0,t]}|B_s| > \frac{\varepsilon e^{-Lt}}{2\sigma \|M\|}\pf \ \leqslant \ \frac{2\sigma \|M\| \sqrt{t} e^{Lt}}{\varepsilon }\,,\]
thanks to Doob's inequality.
\end{proof}

From the previous result, the question of the exit time of $X^\sigma$ before a given fixed time $T$ (independent from $\sigma$) in the low noise regime boils down to the question of the exit time for the deterministic process $X^0$, which is addressed in the next lemma.

 \begin{lem}
\label{roglic}
Under Assumptions~\ref{A} and \ref{D}, 
the process $X^0$  with initial law $m_0$ and solving  the deterministic equation \eqref{eq:init} with $\sigma =0$ satisfies :
\begin{equation*}
    \mathbb P \po X^0_t \in \mathcal D\, \forall t\geqslant 0\pf=1\,.
\end{equation*} 

\end{lem}

\begin{proof}
Let $r>0$ be such that the support of $m_0$ is in the ball $\mathbb B(\lambda,r)\subset \mathcal D$. 
Applying the second point of Proposition~\ref{prop:parallelcoupling} with $\sigma=0$ and using that
\[f(t):=\mathbb E|X^0_t-\lambda|^2 =\mathbb W_2^2 (m_t^0,\delta_\lambda)\]
for all $t\geqslant 0$, we get that $f'(t) \leqslant 0$, and thus $f(t)\leqslant f(0)\leqslant r^2$ for all $t\geqslant 0$. Using now the first point of Proposition~\ref{prop:parallelcoupling}, we get that, almost surely, for all $t\geqslant 0$,
\[\partial_t |X^0_t-\lambda|^2 \leqslant -2\rho |X_t^0-\lambda|^2 + 2\kappa r^2\]
with $|X_0^0 - \lambda|^2 \leqslant r^2$, which implies that $X_t^0 \in \mathbb B(\lambda,r)\subset \mathcal D$ for all $t\geqslant 0$.

\end{proof}

\begin{rem}\label{rem:assumptionInitial} In fact, the same proof works if we only assume that $m_0$ has a compact support in a ball $\mathbb{B}\left(x_0,r'\right)\subset \mathcal D$ where $x_0$ is such that the solution of
\[\partial_t z_t = a(z_t) + b(z_t,\delta_{z_t})\,,\qquad z_0=x_0 \]
stays in $\mathcal D$ for all $t\geqslant 0$.

\end{rem}

\subsection{Proof of Theorem \ref{th:mr}}\label{sec:finalproof}

Fix $a$ and $b$ that satisfy Assumption~\ref{A}, and an initial condition $m_0$. For all $\sigma \geqslant 0$ we consider $(X_t^\sigma)_{t\geqslant 0}$ that solves \eqref{eq:init} where $m_0^\sigma = \mathcal L(X_0^\sigma) =m_0$. Consider $\lambda \in \R^d$ as given by Lemma~\ref{lem:exist_lambda} and, for all $\sigma\geqslant 0$, $(\tilde X_t^\sigma)_{t\geqslant 0} $ that solves \eqref{roserouge} with  $\tilde X_0 = \lambda$.

Now, let $\delta>0$. Let $\xi>0$ be small enough so that $|H_{u,\xi}-H|\leqslant \delta/2$ for $u\in\{i,e\}$. Then, for all $\sigma>0$, according to \ref{K} one has
\begin{align*}
    \mathbb P \po   \tau_\sigma(\mathcal D ) > \exp\po \frac{2}{\sigma^2}\po H  +\delta \pf \pf \pf \ & \leqslant \ \mathbb P \po   \tilde \tau_\sigma(\mathcal D_{e,\xi} ) > \exp\po \frac{2}{\sigma^2}\po H  +\delta \pf \pf \pf   \\
    &\qquad+ \mathbb P\po \tau_\sigma(\mathcal D  ) > \tilde \tau_\sigma(\mathcal D_{e,\xi} ) \pf\,.
\end{align*}
The choice of $\xi$ and \eqref{eq:Kramertilde} imply that the first term of the right-hand side vanishes with $\sigma$, since $H+\delta \geqslant H_{e,\xi}+\delta/2$. Similarly,
   \[ \mathbb P \po   \tau_\sigma(\mathcal D ) < \exp\po \frac{2}{\sigma^2}\po H  -\delta \pf \pf \pf \   \leqslant \  \mathbb P\po \tau_\sigma(\mathcal D  ) < \tilde \tau_\sigma(\mathcal D_{i,\xi} ) \pf + \underset{\sigma\rightarrow 0}o(1)\,.\]
   As a consequence, the proof will be concluded if we show that 
   \[\mathbb P\po \tau_\sigma(\mathcal D  ) > \tilde \tau_\sigma(\mathcal D_{e,\xi} ) \pf + \mathbb P\po \tau_\sigma(\mathcal D  ) < \tilde \tau_\sigma(\mathcal D_{i,\xi} ) \pf \underset{\sigma\rightarrow 0}\longrightarrow 0\,.\]
   We will prove this in two steps: Considering for $T,\varepsilon>0$ the event
   \[\mathcal A_{T,\varepsilon} \ = \ \Big\{\big[\tau_\sigma(\mathcal D)\wedge \tilde \tau_\sigma(\mathcal D_{i,\xi})\big] >T\text{ and } |X_T^{\sigma} - \tilde X_T^\sigma|\leqslant \varepsilon\Big\}\,,\]
   we will prove in Step 1  that there exist $T_0,\varepsilon,\sigma_0>0$ such that for all $\sigma \in (0,\sigma_0]$ and $T\geqslant T_0$,
   \begin{equation}\label{eq:incompatible}
       \mathbb P\po \{\tau_\sigma(\mathcal D  ) > \tilde \tau_\sigma(\mathcal D_{e,\xi} )\} \cap \mathcal A_{T,\varepsilon} \pf + \mathbb P\po \{\tau_\sigma(\mathcal D  ) < \tilde \tau_\sigma(\mathcal D_{i,\xi} )\} \cap \mathcal A_{T,\varepsilon} \pf \ = \  0\,.
   \end{equation} 
   Then in Step 2 we will show that for all $\varepsilon>0$,
   \[\lim_{T\rightarrow+\infty}\liminf_{\sigma\rightarrow 0}\mathbb P \po  \mathcal A_{T,\varepsilon}\pf \ = \ 1\,.\]
The combination of Step 1 and Step 2 concludes the whole proof.   
   \medskip
   
   \noindent \textbf{Step 1.} Remark that, on the event $\mathcal A_{T,\varepsilon}$,  $\tilde \tau_\sigma(\mathcal D_{e,\xi} )>T$ (since $\tilde \tau_\sigma(\mathcal D_{e,\xi} ) \geqslant \tilde \tau_\sigma(\mathcal D_{i,\xi} )$). As a consequence,
   \[\Big(\{\tau_\sigma(\mathcal D  ) > \tilde \tau_\sigma(\mathcal D_{e,\xi} )\}  \cap \mathcal A_{T,\varepsilon} \Big)\ \subset \ \left\{\sup_{t\geqslant T}|X_t^{\sigma}-\tilde X_t^\sigma|\geqslant \xi\right\} \ :=\ \mathcal B_{T,\xi}\,.\]
  Indeed, if $T< s:=\tilde \tau_\sigma(\mathcal D_{e,\xi})<\tau_\sigma(\mathcal D)$, then  $X_s^{\sigma} \in \mathcal D $ while $\tilde Z_s^\sigma \in \partial \mathcal D_{e,\xi}$, so that $|X_s^{\sigma}-\tilde X_s^\sigma|\geqslant \dd (\mathcal D,\mathcal D_{e,\xi}^c) \geqslant \xi$. With a similar argument  we see that 
     \[ \big ( \{\tau_\sigma(\mathcal D  ) < \tilde \tau_\sigma(\mathcal D_{i,\xi} )\}\cap\mathcal A_{T,\varepsilon}\big ) \ \subset \ \mathcal B_{T,\xi}\,.\]
     It only remains to prove that $\mathbb P(\mathcal A_{T,\varepsilon}\cap \mathcal B_{T,\xi})=0$ for $T$ large enough and $\varepsilon,\sigma$ small enough.
     
     From the first part of Proposition \ref{prop:parallelcoupling}, for all $T \geqslant 0$, $\sigma>0$, almost surely
     \[ \sup_{t\geqslant T} |X_t^{\sigma}-\tilde X_t^\sigma|^2 \ \leqslant \ \max\po |X_T^{\sigma}-\tilde X_T^\sigma|^2 \,, \, \frac{\kappa }{\rho}\sup_{t \geqslant T}\mathbb W_2^2(m_t^\sigma,\delta_\lambda)\pf\,,\]
     and thus (using that $\kappa<\rho$),
     \[ \sup_{t\geqslant T} |X_t^{\sigma}-\tilde X_t^\sigma| \ \leqslant \ \max\po  |X_T^{\sigma}-\tilde X_T^\sigma| \,, \, \sup_{t \geqslant T}\mathbb W_2(m_t^\sigma,\delta_\lambda)\pf\,.\]
     In view of Proposition \ref{prop:longtime}, choose $T_0,\sigma_0>0$ such that 
     \[\sup_{t\geqslant T_0} \sup_{\sigma\in(0,\sigma_0]} \mathbb W_2^2\po m_t^\sigma,\delta_\lambda\pf \ \leqslant \  \xi \,.\]
Set $\varepsilon = \xi/2$. Then, for all $T\geqslant T_0$ and all $\sigma\in(0,\sigma_0]$,
   \[\mathcal A_{T,\varepsilon} \ \subset\ \left\{\sup_{t\geqslant T}|X_t^{\sigma}-\tilde X_t^\sigma|\leqslant \xi/2\right\}\subset \mathcal B_{T,\xi}^{\rm{c}}\,,\]
   which concludes the proof of \eqref{eq:incompatible}. 
   
      \medskip
   
   \noindent \textbf{Step 2.} Fix $T,\varepsilon>0$. We bound
  \[\mathbb P \po  \mathcal A_{T,\varepsilon}^c\pf \ \leqslant \ \mathbb P \po  \tau_\sigma(\mathcal D)\leqslant T\pf + \mathbb P\po  \tilde \tau_\sigma(\mathcal D_{i,\xi}) \leqslant T \pf + \mathbb P \po  |X_T^{\sigma} - \lambda| > \varepsilon/2\pf  + \mathbb P \po |\tilde X_T^\sigma-\lambda|>\varepsilon/2\pf\]
  and treat each term separately. The second term vanishes with $\sigma$ according to \eqref{eq:Kramertilde} (see \ref{K}). The last two terms are similar: from the Markov inequality we get
  \[\mathbb P \po  |\tilde X_T^\sigma - \lambda| > \varepsilon/2\pf \ \leqslant \ \frac{4}{\varepsilon^2}\mathbb W_2^2 \po \mathcal L(\tilde  X_T^\sigma) ,\delta_{\lambda}\pf.\]
Notice that, as $\tilde X_0^{\sigma} = \lambda$, we have that for all $t \ge 0$, $\mathcal L(\tilde X_t^{0}) =\delta_\lambda$. Thus, from the second part of Proposition \ref{prop:parallelcoupling} 
we easily get that the right-hand side vanishes with $\sigma$. Similarly, \[\mathbb P \po  |X_T^\sigma - \lambda| > \varepsilon/2\pf \ \leqslant \ \frac{4}{\varepsilon^2}\mathbb W_2^2 \po m_T^\sigma ,\delta_{\lambda}\pf \ \leqslant \ \frac{4}{\varepsilon^2} \po \mathbb W_2 \po m_T^\sigma,m_\infty^\sigma \pf + \mathbb W_2 \po  m_\infty^\sigma  ,\delta_{\lambda}\pf  \pf^2\]
where $m_\infty^\sigma$ is the equilibrium given by  Corollary \ref{cor:equilibrium}.
Now we apply Proposition \ref{prop:longtime} for the following two processes: the process $X^{\sigma}$ and the process solving \eqref{eq:init} with initial condition $m_\infty^\sigma$, so that  $m_t^{\sigma}= m^{\sigma}_\infty$ for all $t\geqslant 0$.
It comes
\begin{equation*}
\mathbb W_2 \po  m_T^\sigma, m_\infty^\sigma\pf \ \leqslant \ Q(T) \mathbb  W_2  \po   m_0^{\sigma} ,m_\infty^{\sigma}\pf \  \leqslant \  Q(T) \po \mathbb  W_2  \po  m_0^{\sigma},\delta_{\lambda}\pf + \mathbb  W_2  \po \mathcal \delta_{\lambda},m_\infty^\sigma)\pf\pf\,. 
\end{equation*}
  From Proposition \eqref{cor:equilibrium_lownoise},
 we see that $\mathbb  W_2  \po \mathcal \delta_{\lambda},\mu_\infty^{\sigma}\pf$ vanishes with $\sigma$, so that
  \[\limsup_{\sigma\rightarrow 0}\mathbb P \po  |X_T^\sigma - \lambda| > \varepsilon/2\pf \ \leqslant \ \frac{4}{\varepsilon^2} Q^2(T) \mathbb  W_2^2  \po  m_0  ,\delta_{\lambda}\pf\,  .\]
  Finally, we bound 
  \[ \mathbb P \po  \tau_\sigma(\mathcal D)\leqslant T\pf  \ \leqslant \ \mathbb P \po \tau_0(\mathcal D_{i,\xi}) \leqslant T\pf +  \mathbb P \po \tau_0(\mathcal D_{i,\xi}) > T \geqslant  \tau_\sigma(\mathcal D)\pf  \]
The first term of the right hand side vanishes with $\sigma$ and, as in Step 1 of the proof, we bound the second term as 
\[\mathbb P \po \tau_0(\mathcal D_{i,\xi}) > T \geqslant  \tau_\sigma(\mathcal D)\pf \ \leqslant \ \mathbb P \po \sup_{s\in[0,T]} |X_s^\sigma - X_s^0| \geqslant \xi \pf\,,\]
which, from Proposition \ref{prop:small_time_coupling},
vanishes with $\sigma$. 

Gathering all these bounds we have obtained that 
  \[\limsup_{\sigma\rightarrow0} \mathbb P \po  \mathcal A_{T,\varepsilon}^c\pf \ \leqslant \ \frac{4}{\varepsilon^2} Q^2(T) \mathbb  W_2^2  \po m_0 ,\delta_{\lambda}\pf \mathbb  W_2^2  \po \mathcal \delta_{\lambda},m_\infty^{\sigma}\pf \ \underset{T\rightarrow \infty}\longrightarrow\  0\,, \]
  which concludes.

\section{Application to particular models }
\label{Sec:applications}

As discussed in  Section~\ref{sec:assum_mainresult}, \ref{A9} is implied by the condition on $a$
\begin{eqnarray}
\label{eq:a_contract2}
\exists \rho>0,\, \forall z,y\in\R^d\,,\qquad (z-y) \cdot   \po a(z)-a(y) \pf & \leqslant & -\rho |z-y|^2 
\end{eqnarray}
 and by taking an interaction $b=\varepsilon \tilde b$ where $\tilde b$ is Lipschitz continuous and $\varepsilon$ is sufficiently small depending on $\rho$ and the Lipschitz constants of $\tilde b$. Moreover, this condition does not involve the  memory kernel $R$. For this reason, we now present, separately, some non-interacting diffusions (solving \eqref{linear_convex}), then some interacting forces $b$, and check that the conditions hold.

\subsection{Some underlying (linear) diffusions}
\label{subsec:ex-linear-drift}

\subsubsection{The overdamped Langevin diffusion} \label{subsec:overdamped}

The overdamped Langevin diffusion is the Markov diffusion on $\R^d$ which solves
\[dX_t = -\nabla V(X_t)dt + \sqrt{2\sigma} dW_t\]
where $V\in\mathcal C^1(\R^d)$, which corresponds to \eqref{eq:init} with $b=0$, $a=-\nabla V$ and $M=I_d$.
 The following is clear:
\begin{prop}
\label{prop:overdamped}
Suppose that $V$ is strongly convex. Then $a=-\nabla V$ satisfies \eqref{eq:a_contract2}.
\end{prop}

\begin{rem}
For $V$ convex, the unique fixed point of $\dot z = -\nabla V(z)$ is the unique minimum of~$V$.
\end{rem}

\subsubsection{The kinetic Langevin diffusion}
\label{subsec:kinetic}

The kinetic Langevin diffusion corresponds to Markov diffusion on $\R^d$ (with $d=2n$ for some $n\geqslant 1$) that solves \eqref{eq:init} with $b=0$ and (decomposing $z=(x,y)\in\R^n\times\R^n$)
\begin{equation}
\label{eq:drift_kinetic}
a(x,y) \ = \ \begin{pmatrix}
  y \\ -\nabla V(x)   - \gamma y
\end{pmatrix}
\,, \qquad M\ = \ \begin{pmatrix}
0 & 0 \\
0 & \sqrt{2\gamma} I_n
\end{pmatrix}
\end{equation}
where $V\in\mathcal C^1(\R^n)$ and $\gamma>0$. The condition \eqref{eq:a_contract2} is not satisfied directly but, as proven in  \cite{Monmarche:contraction}, provided $V$ is convex, its gradient is Lipschitz continuous and the friction $\gamma$ is high enough, then it is satisfied up to a linear change of variable. More precisely, \cite[Proposition 4]{Monmarche:contraction} reads:

\begin{prop}
\label{prop:kinetic}
Suppose that $V\in\mathcal C^2(\R^n)$ and that there exist $\lambda,\Lambda>0$ with
\[\lambda I_n \ \leqslant \ \na^2 V(x) \ \leqslant \ \Lambda I_n\]
for all $x\in\R^n$. Assume furthermore that $\Lambda-\lambda <  \gamma ( \sqrt{\lambda}+\sqrt{\Lambda})$. Then there exists an invertible $d\times d$ matrix $D$ such that $\tilde a$ given by $\tilde a(z) = D^{-1}a(D z)$ for $z\in\R^d$ satisfies \eqref{eq:a_contract2}.
\end{prop} 

\begin{rem}
In the setting of Proposition~\ref{prop:kinetic}, the unique fixed point of $\dot z = a(z)$ is $z_*=(x_*,0)$ where $x_*$ is the unique minimum of $V$ (in other words, $z_*$ is the unique minimum of the Hamiltonian $V(x)+|y|^2/2$).
\end{rem}

\subsubsection{Overdamped Langevin diffusion with coloured noise}

Consider a variation of the overdamped Langevin process where the white noise is replaced by a diffusion process, i.e.
\begin{subequations}
\label{eq: mainSDE}
\begin{align}
dX_t&=-\nabla V(X_t)dt+ B \eta_t\,dt
\\d\eta_t&=F(\eta_t)\,dt+\sigma D dW_t,
\end{align}
\end{subequations}
where $X_t\in \R^d, \eta_t\in \R^n$, $B\in \R^{d\times n}$, $F\in \mathcal C^1(\R^{n},\R^n)$, $D\in \R^{n\times n}$. It corresponds to coefficients
\begin{equation}\label{eq:drift_kinetic2}
a(x,\eta) \ = \ \begin{pmatrix}
-\nabla V(x) + B\eta \\   F(\eta)
\end{pmatrix}
\,, \qquad M\ = \ \begin{pmatrix}
0 & 0 \\
0 & D
\end{pmatrix}
\end{equation}

Important instances of this model include
\begin{enumerate}
\item Scalar OU process: 
\begin{subequations}
\label{eq: OU}
\begin{align}
dX_t&=-\nabla V(X_t)dt+\sqrt{2} \eta_t\,dt
\\d\eta_t&=-\eta_t\,dt+\sqrt{2}\sigma\, dW_t,
\end{align}
\end{subequations}
\item Langevin noise: 
\begin{subequations}
\label{eq: H}
\begin{align*}
dX_t&=-\nabla V(X_t)dt+\sqrt{2} \, y_t\,dt
\\ dy_t&=v_t
\\ dv_t&=-y_t\,dt-\gamma v_t\,dt+\sqrt{2}\sigma\, dW_t.
\end{align*}
\end{subequations}
\end{enumerate}
Notice that, in these two examples, $F$ is linear, so that by considering a rescaled auxiliary variable $\tilde \eta_t=\eta_t/\sigma$ we end up with 
\[d X_t = -\nabla V(X_t)dt + \sqrt 2 \sigma B\tilde \eta_t,\]
where the dynamics of $\tilde \eta$ is independent from $\sigma$. This is more similar to the standard overdamped Langevin with small white noise, however it doesn't fit in this form in the framework of our results  where $\sigma$ is the intensity of the white noise, which is why we wrote it with $\eta$ rather than $\tilde \eta$.

\begin{prop}
\label{prop:colored}
Assume that $V$ is strongly convex, that $F$ satisfies \eqref{eq:a_contract2} and let $a$ be given by \eqref{eq:drift_kinetic2}. Then there exists a diagonal  $d\times d$ matrix $D$ such that $\tilde a$ given by $\tilde a(z) = D^{-1}a(D z)$ for $z\in\R^d$ satisfies \eqref{eq:a_contract2}.
\end{prop}

\begin{rem}
In the case of a Langevin noise, $F$ does not satisfies \eqref{eq:a_contract2} directly but, as discussed in the previous section, we can enforce \eqref{eq:a_contract2} by a linear change of variable. More generally, when $F$ is linear, given by a matrix with spectrum in $\{\lambda\in\mathbb C,\mathfrak{R}e(\lambda)<0\}$, it is known that \eqref{eq:a_contract2} holds up to a linear change of variable.
\end{rem}

\begin{proof}
 Assume that \eqref{eq:a_contract2} holds for $F$ for some $\rho$ and that $V$ is $\rho'$-strongly convex. Let  $\tilde a(x,\eta)=a(c x,\eta)$ with $c=|B|^2/(\rho\rho')$. Then, for all $z=(x,\eta),z'=(x',\eta')\in\R^d\times\R^n$,
\begin{align*}
(z-z')\cdot \po \tilde a(z) - \tilde a(z') \pf & = (x-x')\cdot\po \na V(c x')-\na V(c x) + B(\eta-\eta')\pf +(\eta-\eta')\po F(\eta)-F(\eta')\pf\\
&\leqslant -c\rho' |x-x'|^2 + |B||x-x'||\eta-\eta'| - \rho |\eta-\eta'|^2\\
& \leqslant -\frac{c\rho'}{2}|x-x'|^2 -\frac{\rho}{2}|\eta-\eta'|^2\,.
\end{align*}

\end{proof}

\subsubsection{The generalized Langevin diffusion}

The generalized Langevin diffusion corresponds to Markov diffusion on $\R^d$ (with $d=2n+p$ for some $n,p\geqslant 1$) that solves \eqref{eq:init} with $b=0$ and (decomposing $z=(x,y,w)\in\R^n\times\R^n\times \R^p$)
\begin{equation}\label{eq:drift_kinetic3}
a(x,y,w) \ = \ \begin{pmatrix}
  y \\ -\nabla V(x)    \\
 0
\end{pmatrix} - \gamma \begin{pmatrix}
0 & 0 & 0 \\
0 & B_{11} & B_{12}\\
0 & B_{21} & B_{22}
\end{pmatrix} \begin{pmatrix}
x \\ y \\ w
\end{pmatrix}
\,, \quad M\ = \ \sqrt{\gamma} \begin{pmatrix}
0 & 0 & 0   \\
0 & \Sigma_{11} & \Sigma_{12}\\
0 & \Sigma_{21} & \Sigma_{22}
\end{pmatrix} 
\end{equation}
where $V\in\mathcal C^1(\R^n)$,  $\gamma>0$ and the $\Sigma_{i,j}$'s and $B_{i,j}$'s are constant matrices. These processes arise in Monte Carlo method \cite{Leimkuhler,Pavliotis,Monmarche:contraction} and in effective dynamics problem in molecular dynamics \cite{Stella,CBP:GLE}. The fluctuation-dissipation relation is said to be satisfied if
\[\Sigma^T \Sigma \ = \ B^T+B\qquad \text{where} \qquad  B \ = \ \ \begin{pmatrix}
  B_{11} & B_{12}\\
 B_{21} & B_{22}
\end{pmatrix}, \qquad  \Sigma  \ = \ \ \begin{pmatrix}
  \Sigma_{11} & \Sigma_{12}\\
 \Sigma_{21} & \Sigma_{22}
\end{pmatrix}\,.\]
In that case, the invariant measure of the process is the probability density proportional to $\exp(-[V(x)+|y|^2/2+|w|^2/2]/\sigma^2)$. Classical examples are $K^{th}$ order Generalized Langevin processes for $K\geqslant 3$, which corresponds to $p=(K-2)n$, i.e. the total dimension is $d=Kn$, and, decomposing $B$ in $n\times n$ blocks,
 \[
 B =  \gamma \begin{pmatrix}
0  & -I_n & 0 & \dots & 0 \\
 I_n & 0 & -I_n & \ddots & \vdots \\
0 & \ddots & \ddots & \ddots &   0 \\
\vdots & \ddots &   I_n & 0 & - I_n\\
0 &\dots  & 0 &  I_n &   I_n
\end{pmatrix}\,.\]
In other words, the process solves
\begin{eqnarray*}
d X_t & =& Y_t dt\\
d Y_t & = & -\nabla V(X_t) dt +\gamma Z^{(1)}_t dt\\
d Z^{(1)}_t & = & \gamma \po Z^{(2)}_t - Y_t\pf dt\\
d Z^{(2)}_t & = & \gamma \po Z^{(3)}_t - Z^{(1)}_t\pf dt\\
& \vdots & \\
d Z^{(K-3)}_t & = & \gamma \po Z^{(K-2)}_t - Z^{(K-3)}_t\pf dt\\
d Z^{(K-2)}_t & = & -\gamma \po Z^{(K-2)}_t + Z^{(K-3)}_t\pf dt + \sqrt{2\gamma} dW_t\,.\\
\end{eqnarray*}
Here, the SDE is very degenerated, since a $d$-dimensional noise drives a $Kd$-dimensional process. See \cite{Pavliotis,Leimkuhler} and references for more detailed discussions and more examples.

Similarly to the kinetic Langevin case, assuming that $V$ is strongly convex with bounded Hessian and that the friction is large enough, it can be proven that, up to a linear change of variables, $a$ given by \eqref{eq:drift_kinetic3} satisfies \eqref{eq:a_contract2}, see \cite[Proposition 5]{Monmarche:contraction} for details.

\begin{rem} The unique fixed point of $\dot z = a(z)$ where $a$ is given by \eqref{eq:drift_kinetic3} is $z_*=(x_*,0,0)$ where $x_*$ is the unique minimum of $V$ (in other words, $z_*$ is the unique minimum of the Hamiltonian $V(x)+|y|^2/2+|w|^2/2$).
\end{rem}

\subsection{Examples of interactions}
\label{subsec:ex-interact}

\subsubsection{Interaction potential}\label{sec:interactionpotential}
One of the most classical interaction mechanism is given by
\begin{equation}
    \label{eq:bpotential}
    b(z,\mu) \ = \ A \int_{\R^d} \na_z W(z,y)   \mu(\dd y)
\end{equation}
where $W\in \mathcal C^1(\R^d\times \R^d)$ and $A$ is a $d\times d$ constant matrix. For instance, in the case of the overdamped Langevin diffusion (see Section~\ref{subsec:overdamped}), typically $A=-I_d$, and in the case of the kinetic Langevin diffusion  (with $d=2n$, see Section~\ref{subsec:kinetic}),
\[A \ = \ \begin{pmatrix}
0 & 0 \\
-I_n & 0 
\end{pmatrix}\]
and $W$ is given by $W((x,y),(x',y')) = \tilde W(x,x')$ for some $\tilde W\in\mathcal C^1(\R^n\times \R^n)$.

\begin{prop}
Suppose that $\na_z W$ is a Lipschitz function. Then $b$ given by \eqref{eq:bpotential} satisfies Assumption~\ref{A}.
\end{prop}

\begin{proof}
Let $L>0$ be such that
\[|\na_z W(z,y)-\na_z W(z',y')| \ \leqslant \ L\po |z-z'|+|y-y'|\pf \]
for all $z,z',y,y'\in\R^d$.

Let $z,y\in\R^d$ and $\mu,\nu\in\mathcal P_2(\R^d)$. 
Let $(Z,Y)$ be a $\mathbb W_2$-optimal coupling of $\mu$ and $\nu$. Then
\begin{align*}
    |b(z,\mu)-b(y,\nu)| \ &= \ \left|\mathbb E \po A\na_z W(z,Z) - A\na_z W(y,Y)\pf\right|\\
    & \leqslant\ \|A\| L \po |y-z| + \mathbb E\po |Z-Y|\pf \pf  \ \leqslant \ \|A\|L \po |y-z| + \mathbb W_2(\nu,\mu)\pf\,. 
\end{align*}
\end{proof}

\begin{rem}
In most classical cases, $W(z,y)=\tilde W(z-y)$ with an even $\tilde W \in \mathcal C^1(\R^d)$.  Then, $b(z,\delta_z) = 0$ for all $z\in\R^d$, and thus $a(\lambda)+b(\lambda,\delta_\lambda)=0$ if and only if $\lambda$ is a fixed point of $\dot z = a(z)$.
\end{rem}

\subsubsection{The Adaptive Biasing Potential algorithm}

The standard adaptive biasing algorithms used in molecular dynamics do not use a repulsive bias which is isotropic in the whole space. Rather, they tend to bias only a small dimensional part of the space, described by so-called reaction coordinates or collective variables. Reaction coordinates are given by $\xi:\R^d\rightarrow \R^p$ with $p\ll d$ (typically, in a full-atom simulation, $d=3N$ with $N$ the number of atoms which may be of order $10^6$, while $p=1$ or $2$). For instance, if $z$ is the positions of $N$ atoms, $\xi(z)$ may be the distance between two particular atoms. We refer to \cite{LelievreABF} for more details and motivations of the use and design of reaction coordinates, in particular for enhanced sampling with self-biasing processes. A key point is that good reaction coordinates are meant to encode most of the metastability of the system, namely, if all the statistically representative values of $\xi(z)$ along a trajectory have been visited then  this should also be the case for $z$. For this reason, adaptive algorithms based on reaction coordinates are designed so that, at stationarity, the law of $\xi(z)$ is unimodal. This is the case for algorithms such as the metadynamics \cite{Metadyn1,Metadyn2,Metadyn3}, Adaptive Biasing Force (ABF) \cite{LelievreABF} or Adaptive Biasing Potential (ABP) \cite{BBM2020} methods.

Let us focus on the ABP method here. For a small  $\varepsilon>0$, consider a Gaussian kernel $K_\varepsilon(x)=e^{-|x|^2/(2\varepsilon)}/(2 \pi \varepsilon)^{p/2}$ on $\R^p$. If $Z$ is a random variable on $\R^d$ with law $\mu$, then a smooth approximation for the law of $\xi(Z)$ is given by the density
\[\rho_\mu(x) = \int_{\R^d} K_\varepsilon\po x-\xi(z)\pf \mu(dz)\,.\]
The ABP algorithm (or more precisely its mean field limit) can then be written as the self-interacting process \eqref{eq:init} with interaction given by
\[b(z,\mu) = \omega A\nabla \ln (\varepsilon' + \rho_\mu \circ \xi)(z)\,,\]
where the matrix $A$ is as in the previous section (depending whether we consider an overdamped or a kinetic Langevin dynamics), $\varepsilon'>0$ is another small regularisation parameter and $\omega\in[0,1]$ parametrizes the strength of the bias. It is straightforward to check that $b$ is Lipschitz continuous. However, it should be noticed that the Lipschitz constants  depends on $\varepsilon,\varepsilon'$ and blow up as these regularization parameters vanish. Due to the condition $\rho>\kappa$ in \ref{A9}, for a fixed value of $\omega$, applying our results to the ABP algorithm prevents $\varepsilon$ and $\varepsilon'$ to be too small, which is a limitation in term of practical application of the algorithm.

\subsubsection{Interaction between two different species}

In \cite{DT3}, a two-species model in the form of a coupled system of nonlinear stochastic differential equations has been studied. It corresponds to the function $b$ on $\R^d\times \mathcal P_2(\R^d)$ with $\R^d=\R^n\times\R^n$ of the form
\[b((x,y),\nu) = \begin{pmatrix}
\int_{\R^d} \po b_{11}(x-x') + b_{12}(x-y')\pf \nu(\dd x',\dd y')\\
\int_{\R^d} \po b_{21}(y-x') + b_{12}(y-y') \pf \nu(\dd x',\dd y')
\end{pmatrix}\]
where the functions $b_{ij}:\R^n\rightarrow\R^n$ are Lipschitz continuous. It follows that $b$ is Lipschitz continuous, as in Section~\ref{sec:interactionpotential}.

\subsection{Kramers' law for linear processes}
\label{Sec:LDP_linear}

In this section, we gather some classical results which can be used to check Assumption \ref{K}. Standard references for large deviation principles and Kramers' law are the books of Freidlin and Wentzell \cite{FW} and of Dembo and Zeitouni \cite{DZ}.

Following \cite{DZ}, we now define the quasi-potential associated to the process \eqref{roserouge}. First, for $x,y\in\R^d$ and $t>0$, consider the cost function 
\[V(x,y,t) = \inf \frac14\int_0^t |u(s)|^2 ds\,,\]
where the infimum runs over all $u\in L^2([0,t])$ such that $z_t=y$ where $(z_s)_{s\in[0,t]}$ is the solution of $z_s=x+\int_0^s(a(z_w)+b(z_w,\delta_\lambda)+\Sigma u_w) dw,\ s\in[0,t]$. The quasi-potential for \eqref{roserouge} is then defined as
\[\overline{V}(x) = \inf_{t>0} V(\lambda,x,t)\,.\]
Then, in our framework, \cite[Theorem 5.7.11]{DZ} reads as follows.

\begin{thm}\label{thm:DZ}
Under \ref{A}, assume furthermore the following:
\begin{enumerate}
    \item The domain $\mathcal D$ is open, bounded and positively invariant for the vector field $a+b(\cdot,\delta_\lambda)$.
    \item $H:=\inf_{z\in\partial \mathcal D} \overline{V}(z) < \infty$.
    \item There exists $M>0$ and a function $T:\R_+\rightarrow \R_+$  with $T(s)\rightarrow 0$ as $s\rightarrow 0$ such that, for all $\varepsilon>0$ small enough and all $x,y\in\R^d$ with $|x-z|+|y-z|\leqslant \varepsilon$ for some $z\in\partial\mathcal D\cup\{\lambda\}$, there exists a function $u :[0,T(\varepsilon)]\rightarrow \R^d$ with $\|u\|_{L^2} \leqslant M$ such that $z_{T(\varepsilon)}=y$ where $z$ solves $z_s=x+\int_0^s(a(z_w)+b(z_w,\delta_\lambda)+\Sigma u_w) dw,\ s\in[0,T(\varepsilon)]$.
\end{enumerate}
Then, for all $x\in \mathcal D$ and all $\delta>0$:
\begin{equation}\label{eq:K}
\mathbb P_x \po \exp\po \frac{2}{\sigma^2}\po H -\delta \pf\pf  \leqslant \tilde \tau_\sigma(\mathcal D) \leqslant \exp\po \frac{2}{\sigma^2}\po H +\delta \pf \pf \pf \underset{\sigma\rightarrow0}\longrightarrow 1\,.
\end{equation}
\end{thm}

Indeed, $\mathcal D$ being positively invariant for $a+b(\cdot,\delta_\lambda)$, from Lemma~\ref{cor:deterministic_flow}, we get that necessarily $\lambda\in\mathcal D$ and all solutions of $\dot z=a(z)+b(z,\delta_\lambda)$ converges to $\lambda$, which is required to apply  \cite[Theorem 5.7.11]{DZ}. The third condition of the theorem is a controllability assumption, which is in particular always satisfied if $\Sigma$ is non-singular, see \cite[Exercise 5.7.29]{DZ}.

In order to get \A{K}, we need to apply Theorem~\ref{thm:DZ} to the domains $\mathcal D_{u,\xi}$. In particular these approximations of $\mathcal D$ should be positively invariant, which can be done following the construction of 
\cite[Definition 2.1, Proposition 2.2]{EJP}. The last thing to check is that $H_{u,\xi} = \inf_{z\in\partial \mathcal D_{u,\xi}} \overline{V}(z)$ converge to $H$ as $\xi$ vanishes, for $u\in\{e,i\}$. This follows from the continuity of $\bar V$ in the vicinity of $\partial D$, which is a consequence of the third assumption of Theorem~\ref{thm:DZ}, as proven in \cite[Lemma 5.7.8]{DZ}.

\medskip

 In the kinetic Langevin case (as in Section~\ref{subsec:kinetic}) the third  assumption of Theorem~\ref{thm:DZ} does not hold, and one would like to consider unbounded domains $\mathcal D=\mathcal D'\times\R^n$ where $\mathcal D'$ is bounded (i.e.  we are interested in the exit time of the position of the kinetic process from a given domain). In the case where
 \begin{equation}\label{eq:KramLangevin}
a(x,y)+b((x,y) ,\delta_\lambda) = \begin{pmatrix}
    y \\
    -c(x) - \gamma y
    \end{pmatrix}\,,\qquad 
    \Sigma = \sqrt{2\gamma}\begin{pmatrix}
    0 & 0 \\
    0 & I_n
    \end{pmatrix}\,,
     \end{equation}
     and $\mathcal D'$ is a smooth bounded domain with $c(x)\cdot \mathrm{n}(x)<0$ for all $x\in \partial\mathcal D'$ where $\mathrm{n}$ is the exterior normal to $\partial \mathcal D'$ (which, under \ref{A}, necessarily implies that $\lambda\in\mathcal D'\times\R^n$; besides,  necessarily $\lambda=(\lambda',0)$ for some $\lambda'\in\R^n$),
    it is proven in \cite{Freidlin2004} that, in particular if the initial condition is $\lambda$, then \eqref{eq:K} holds with
    \[H = \inf_{x\in\partial\mathcal D'} S(x) \]
    where, for $x\in\R^n$,
    \begin{align*}
     S(x)&=\inf\{I_{[0T]}(\varphi):~ \varphi_0=x_*,~\varphi_T=x,~\dot{\varphi}_0=0,~T\geq 0, ~\varphi\in C_{0T}\}\\
I_{[0,T]}(\varphi)&=\begin{cases}
\frac{1}{4}\int_0^T \left|\Ddot{\varphi}_t+\gamma \dot{\varphi}_t+c(\varphi_t)\right|^2\, dt,\quad\text{if}\,\,\dot{\varphi}\text{ is abs. cont.}\\
+\infty\quad\text{otherwise}.
\end{cases}    
\end{align*}
Finally, notice that if the condition $c\cdot \mathrm{n}<0$ is satisfied at the boundary of $\mathcal D'$, then  families of domains $\mathcal D'_{u,\xi}$ for $u\in\{e,i\}$ and $\xi\geqslant 0$ are easily constructed to satisfy the same condition and to meet the requirements of \ref{K}. Moreover, it is not difficult to see that $S$ is continuous, from which $H_{i,\xi}:= \inf_{x\in\mathcal D_{u,\xi}} S(x) \rightarrow H$ as $\xi$ vanishes. As a consequence, in the framework presented here, \ref{K} holds.

\medskip

To conclude, let us recall two classical cases where the quasi-potential is explicit.

\begin{enumerate}
    \item \textbf{The equilibrium elliptic case}. If $a+b(\cdot,\delta_\lambda)=-\na U$ for some $U\in \mathcal C^2(\R^d)$ and $\Sigma=I_d$, then it is well-known (see e.g. \cite{FW}) that $\overline{V}=U-U(\lambda)$ (notice that $\lambda$ is necessarily the unique global minimum of $U$ since all solutions to $\dot z= -\na U(z)$ converges to $\lambda$). In this case, in Theorem~\ref{th:mr}, $H=\inf_{z\in\partial\mathcal D} U-U(\lambda)$. 
    \item \textbf{The equilibrium kinetic  case}. If $d=2n$ and $a$ and $\Sigma$ are given by \eqref{eq:KramLangevin} with $c = \na U$ for some $U\in\mathcal C^2$,
  then $S=U-U(\lambda')$, as proven in \cite{Freidlin2004}. Again, in this case, if $\mathcal D=\mathcal D'\times\R^n$ with $-\na U\cdot \mathrm n<0$ at the boundary of $\mathcal D'$ then, in Theorem~\ref{th:mr}, $H=\inf_{z\in\partial\mathcal D'} U-U(\lambda')$. 
\end{enumerate}

\subsubsection*{Acknowledgments}

This work is supported by the  French ANR grant METANOLIN (ANR-19-CE40-0009). P. Monmarch{\'e} acknowledges financial support by the  French ANR grant SWIDIMS (ANR-20-CE40-0022). The research of M. H. Duong was supported by the EPSRC Grant EP/V038516/1. M. Toma\v sevi\'c was supported by \textit{Fondation Math\'ematique Jacques Hadamard}.

\bibliographystyle{plain}
\bibliography{biblio.bib}

\begin{thebibliography}{10}

\bibitem{Metadyn3}
Alessandro Barducci, Giovanni Bussi, and Michele Parrinello.
\newblock Well-tempered metadynamics: a smoothly converging and tunable
  free-energy method.
\newblock {\em Physical review letters}, 100(2):020603, 2008.

\bibitem{BB}
Michel {Bena{\"\i}m} and Charles-Edouard {Br{\'e}hier}.
\newblock {Convergence analysis of adaptive biasing potential methods for
  diffusion processes}.
\newblock {\em Communications in Mathematical Sciences}, 17(1):81--130, 2019.

\bibitem{BBM2020}
Michel {Bena{\"\i}m}, Charles-Edouard {Br{\'e}hier}, and Pierre
  {Monmarch{\'e}}.
\newblock {Analysis of an Adaptive Biasing Force method based on
  self-interacting dynamics}.
\newblock {\em to appear in Electronic Journal of Probability}, 2020.

\bibitem{CBP:GLE}
Michele Ceriotti, Giovanni Bussi, and Michele Parrinello.
\newblock Colored-noise thermostats à la carte.
\newblock {\em Journal of Chemical Theory and Computation}, 6(4):1170--1180,
  2010.

\bibitem{DZ}
Amir Dembo and Ofer Zeitouni.
\newblock {\em Large deviations techniques and applications}, volume~38 of {\em
  Stochastic Modelling and Applied Probability}.
\newblock Springer-Verlag, Berlin, 2010.
\newblock Corrected reprint of the second (1998) edition.

\bibitem{DT3}
Manh~Hong Duong and Julian Tugaut.
\newblock Coupled {M}c{K}ean-{V}lasov diffusions: wellposedness, propagation of
  chaos and invariant measures.
\newblock {\em Stochastics}, 92(6):900--943, 2020.

\bibitem{ELM}
Virginie Ehrlacher, Tony Leli{\`e}vre, and Pierre Monmarch{\'e}.
\newblock {Adaptive force biasing algorithms: new convergence results and
  tensor approximations of the bias}.
\newblock working paper or preprint, July 2020.

\bibitem{Fortetal}
Gersende Fort, Benjamin Jourdain, Tony Leli{\`e}vre, and Gabriel Stoltz.
\newblock Convergence and efficiency of adaptive importance sampling techniques
  with partial biasing.
\newblock {\em J. Stat. Phys.}, 171:220--268, 2018.

\bibitem{Freidlin2004}
Mark~I. Freidlin.
\newblock Some remarks on the {S}moluchowski-{K}ramers approximation.
\newblock {\em Journal of Statistical Physics}, 117:617--634, 2004.

\bibitem{FW}
Mark~I. Freidlin and Alexander~D. Wentzell.
\newblock {\em Random perturbations of dynamical systems}, volume 260 of {\em
  Grundlehren der Mathematischen Wissenschaften [Fundamental Principles of
  Mathematical Sciences]}.
\newblock Springer-Verlag, New York, second edition, 1998.
\newblock Translated from the 1979 Russian original by Joseph Sz\"{u}cs.

\bibitem{HolleyKusuoakaStroock}
Richard~A. Holley, Shigeo Kusuoka, and Daniel~W. Stroock.
\newblock Asymptotics of the spectral gap with applications to the theory of
  simulated annealing.
\newblock {\em J. Funct. Anal.}, 83(2):333--347, 1989.

\bibitem{Metadyn1}
Benjamin Jourdain, Tony Leli{\`e}vre, and Pierre-Andr{\'e} Zitt.
\newblock Convergence of metadynamics: discussion of the adiabatic hypothesis.
\newblock {\em arXiv preprint arXiv:1904.08667}, 2019.

\bibitem{Metadyn2}
Alessandro Laio and Michele Parrinello.
\newblock Escaping free-energy minima.
\newblock {\em Proceedings of the National Academy of Sciences},
  99(20):12562--12566, 2002.

\bibitem{Leimkuhler}
Benedict {Leimkuhler} and Matthias {Sachs}.
\newblock {Efficient Numerical Algorithms for the Generalized Langevin
  Equation}.
\newblock {\em arXiv e-prints}, page arXiv:2012.04245, December 2020.

\bibitem{LelievreMetastable}
Tony Leli{\`e}vre.
\newblock Two mathematical tools to analyze metastable stochastic processes.
\newblock In {\em Numerical mathematics and advanced applications 2011}, pages
  791--810. Springer, Heidelberg, 2013.

\bibitem{LelievreABF}
Tony Leli{\`e}vre, Mathias Rousset, and Gabriel Stoltz.
\newblock Long-time convergence of an adaptive biasing force method.
\newblock {\em Nonlinearity}, 21(6):1155--1181, 2008.

\bibitem{LelievreFreeEnergy}
Tony Leli\`evre, Mathias Rousset, and Gabriel Stoltz.
\newblock {\em Free energy computations: A mathematical perspective}.
\newblock Imperial College Press, 2010.

\bibitem{ActaNumericaLS}
Tony Leli{\`e}vre and Gabriel Stoltz.
\newblock {Partial differential equations and stochastic methods in molecular
  dynamics}.
\newblock {\em {Acta Numerica}}, 25:681--880, May 2016.

\bibitem{Monmarche:contraction}
Pierre {Monmarch{\'e}}.
\newblock {Almost sure contraction for diffusions on $\mathbb R^d$. Application
  to generalised Langevin diffusions}.
\newblock {\em arXiv e-prints}, page arXiv:2009.10828, September 2020.

\bibitem{Pavliotis}
Michela Ottobre and Grigorios~A. Pavliotis.
\newblock Asymptotic analysis for the generalized {L}angevin equation.
\newblock {\em Nonlinearity}, 24(5):1629--1653, 2011.

\bibitem{Stella}
Lorenzo Stella, Chris~D. Lorenz, and Lev Kantorovich.
\newblock Generalized langevin equation: An efficient approach to
  nonequilibrium molecular dynamics of open systems.
\newblock {\em Phys. Rev. B}, 89:134303, Apr 2014.

\bibitem{SV}
Daniel~W. Stroock and S.~R.~Srinivasa Varadhan.
\newblock {\em Multidimensional diffusion processes}, volume 233 of {\em
  Grundlehren der Mathematischen Wissenschaften [Fundamental Principles of
  Mathematical Sciences]}.
\newblock Springer-Verlag, Berlin-New York, 1979.

\bibitem{EJP}
Julian Tugaut.
\newblock Exit problem of {M}c{K}ean-{V}lasov diffusions in convex landscapes.
\newblock {\em Electron. J. Probab.}, 17:no. 76, 26, 2012.

\bibitem{ECP}
Julian Tugaut.
\newblock A simple proof of a {K}ramers' type law for self-stabilizing
  diffusions.
\newblock {\em Electron. Commun. Probab.}, 21:Paper No. 11, 7, 2016.

\end{thebibliography}

\end{document}